\documentclass[12pt]{article}
\usepackage{amssymb,latexsym}
\usepackage{amsmath}
\usepackage{tikz}
\usepackage{verbatim}
\usetikzlibrary {positioning}
  \usetikzlibrary{arrows,topaths}
       \usetikzlibrary{decorations,backgrounds,snakes}
\usepackage{rawfonts}
\usepackage{amsmath, psfrag}
\usepackage{amsfonts}
\usepackage{amsthm}
\usepackage{lineno,hyperref}

\usepackage{verbatim}
\usepackage {tikz}
\usetikzlibrary {positioning}
\definecolor {processblue}{cmyk}{0.96,0,0,0}

\usepackage{latexsym}
\newtheorem{theorem}{Theorem}
\newtheorem{lemma}[theorem]{Lemma}

\def\qed{\ifhmode\unskip\nobreak\fi\quad\ifmmode\Box\else$\Box$\fi}

\begin{document}
\title{Escaping from the corner of a grid\\
by edge disjoint paths}
\author{{\sl Adam S. Jobson} \\ University of Louisville\\ Louisville, KY 40292
\and {\sl Andr\'e E. K\'ezdy}\\ University of Louisville\\ Louisville, KY 40292
\and {\sl Jen\H{o} Lehel} \\ University of Louisville\\ Louisville, KY 40292
\\
and 
\\
Alfr\'ed R\'enyi Mathematical Institute, \\
Budapest, Hungary}

\maketitle

\begin{abstract}
Let $Q$ be a finite subgraph of the integer grid $G$ in the plane, and let $T$ be a set of pairs of distinct vertices in $G$, called `terminal pairs'. Escaping a subset $X\subset T\cap Q$ from $Q$ means finding edge disjoint paths from the terminals in $X$ into distinct vertices of a set $L$ in the  boundary of $Q$. Here we prove several lemmas for the cases where $Q$ is a $3\times 3$ grid, $L$ is the union of a vertical and horizontal boundary line of $Q$, furthermore, $T$ is a set of four terminal pairs in $G$ such that $|T\cap Q|\geq 5$. These lemmas are applied in \cite{infty} and complete the proof that the Cartesian product of two (one way) infinite paths has path-pairability number four.
\end{abstract}

\section{Introduction}

Finding disjoint paths in grids emerges in several practical  applications, among others the point-to-point delivery problem  \cite{LMS},
 the reconfiguration problem for VLSI  arrays \cite{rec}. From algorithmic point of view the complexity of answering those problems   is hard in general. The results proved here help solve certain extremal problems on the linkage of terminals proposed in \cite{CSFGYLS} when modeling telecommunication networks. 

For fixed $k$, a graph $G$ is {\it $k$-path-pairable}, if for any set of $k$ disjoint pairs of vertices  $s_i,t_i$, $1\leq i\leq k$, called {\it terminals}, there exist pairwise edge-disjoint $s_i,t_i$-paths in $G$. The {\it path-pairability number}, denoted $pp(G)$,  is the largest $k$ such that $G$ is $k$-path-pairable. A summary of early results concerning
path-pairability was given in \cite{F}. The parameter $pp(G)$ was investigated recently in  \cite{infty}, \cite{pxp}, \cite{66}, for  finite and infinite grid graphs $G$ (equivalently for the Cartesian products of finite or infinite paths). 

In \cite{infty} it is proved that the path-pairability number of the positive integer quadrant in the Euclidean plane, $\Bbb{R}^2$, is four. In the proof we needed to discuss the linkage of four terminal pairs when five or more of these terminals are located in the $3\times 3$ corner of the integer quadrant.
In the present note we give the proof of those `escaping' lemmas  applied in \cite{infty}.\\

Let $G$ be the Cartesian product of two (one-way) infinite paths with 
vertices $(i,j)$, $i,j\in \{1,2,\dots\}$, with an edge between 
$(i,j)$ and $(p,q)$ if and only if $|p-i|+|q-j|=1$. We represent $G$ as a matrix of its vertices arranged in rows
 $A(i)$, $i=1,2,\dots$, and columns $B(j)$, $j=1,2,\dots$, each being an infinite path. Let $Q\subset G$ be the $3\times 3$ corner of $G$ induced by 
 $\{(i,j)\mid 1\leq i,j\leq 3\}$.  
 
The notation $H\subset G$ means that $H$ is a subgraph,  for a vertex set $S$, the notation $H-S$ is interpreted as the subgraph obtained by the removal of $S$ and the incident edges from $H$, furthermore, $x\in H$ means that $x$ is a vertex of $H$.\\

Let $T
\subset G$ be a set of $8$ distinct vertices in $G$, called {\it terminals}, partitioned into the four terminal pairs  $\pi_i=\{s_i,t_i\}$, $1\leq i\leq 4$.
To prove the $4$-path-pairability of $G$, one must find a linkage for $\pi_i$,  $1\leq i\leq 4$, that is a set of edge disjoint $s_i,t_i$-paths $P_i\subset G$. One difficulty is arising when five or more terminals are packed in $Q$ and the pairing of those terminals requires leaving $Q$ towards locations outside $Q$.

We will say that the terminals in $Q$ can `escape' from $Q$
 (thus `move' into $G-Q$) if there are edge disjoint paths from the terminals into distinct vertices of the union $L$ of the horizontal  boundary line $A=A(3)\cap Q$ and the vertical boundary line $B=B(3)\cap Q$.
The notation $\|S\|=|T\cap S|$ is the number of terminals in the subgraph $S\subseteq G$.

In section \ref{proofs} we prove escaping lemmas for the cases when $5\leq t\leq 8$ terminals of $T$ lie in $Q$; some pairs are linked in $Q$, and the  unpaired terminals are `mated'  into distinct `escape' vertices of the boundary $L=A\cup B$ using edge disjoint escape paths. Section \ref{tools} introduces a few tools used in the proofs for escaping from $Q$.

\section{Tools}\label{tools}

Escaping from $Q$, as we will handle here, requires three kinds of operation: finding a linkage of one or two terminal pairs, `moving' a terminal into a `mate' at a suitable location, and `shifting' a terminal along the boundary line $L$.

Finding a linkage for two pairs are facilitated using the property of a graph being `weakly $2$-linked' (see in \cite{T}),
and by introducing the concept of a `frame' (see \cite{66}). The first tool is stated in the next lemma (its simple proof is omitted). 

\begin{lemma}
\label{w2linked}
The subgraph $H=Q$ or $H=Q-(3,3)$ of $G$ is weakly $2$-linked, that is, for every not necessarily distinct vertices $u_1,v_1,u_2,v_2\in H$ there exist edge disjoint $u_i,v_i$-paths in $H$, for $i=1,2$.
 \qed
\end{lemma}

The linkage of two pairs can be obtained using a frame.
  Let $C\subset G$ be a cycle, let $x\in C$; and take edge disjoint paths from a member
  of $\pi_j$ to $x$, for $j=1,2$, not using edges of $C$. Then we say that
  the subgraph of the union of $C$ and the two paths to $x$ define a {\it frame} $[C,x]$ for $\pi_1,\pi_2$. A frame
  $[C,x]$
  helps find a linkage
  for the pairs $\pi_1$ and $\pi_2$, it is enough to mate the other members of the pairs anywhere in $G$ into any vertices of $C$  using mating paths which are edge disjoint from $[C,x]$ and from each other. An example of a frame in $G$ is seen in Fig.\ref{frame}.
  
   	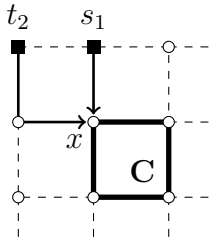
\begin{figure}[htp]
\tikzstyle{T} = [rectangle, minimum width=.1pt, fill, inner sep=2.5pt]
\tikzstyle{B} = [circle, draw=black!, minimum width=1pt, fill=white, inner sep=1.5pt]
\tikzstyle{txt}  = [circle, minimum width=1pt, draw=white, inner sep=0pt]
\tikzstyle{Wedge} = [draw,line width=.7pt,-,black!100]
\tikzstyle{M} = [circle, draw=black!, minimum width=1pt, fill=white, inner sep=1pt]
\begin{center}

  		\begin{tikzpicture}
 \draw[dashed]   (1,0)--(0,0) -- (0,1) (0,2)--(2,2) --(2,1);
  \foreach \x in {0,1,2}\draw[dashed](\x,0)--(\x,-.6){};   
  \foreach \y in {0,1,2}\draw[dashed](2,\y)--(2.6,\y){};      
   \draw[line width=2pt]    (2,0) -- (2,1) -- (1,1)--(1,0)--(2,0);
  \draw[->,line width=1pt] (0,2) -- (0,1)--(.9,1);
        \draw[->,line width=1pt]  (1,2)--(1,1.1);   
         \foreach \x in {0,1,2}\foreach \y in {0,1,2}\node[B]()at(\x,\y){};                        
                  
            \node[T,label=above:$t_2$] (t2) at (0,2) {};                                    
            \node[T,label=above:$s_1$] (s1) at (1,2) {};           
               \node[txt] () at (.75,.75) {$x$};  
            
                \node[txt] () at (1.65,.35) {\bf C};     
   
     	\end{tikzpicture}	
	
\end{center}
	\caption{A frame $[C,x]$ for $\pi_1, \pi_2$}
	\label{frame}

	\end{figure}

Given two vertices $u,v\in L$, a subgraph $CL(u,v)\subset Q$ will be called a {\it clip} on $u,v$, if for any two terminals $x,y\in CL(u,v)$ it contains two edge disjoint mating paths, 
\begin{figure}[htp]
  \tikzstyle{T} = [rectangle, minimum width=.1pt, fill, inner sep=2.5pt]
\tikzstyle{B} = [circle, draw=black!, minimum width=1pt, fill=white, inner sep=1.5pt]
\tikzstyle{txt}  = [circle, minimum width=1pt, draw=white, inner sep=0pt]
\tikzstyle{M} = [circle, draw=black!, minimum width=1pt, fill=white, inner sep=1pt]
\begin{center} 
\begin{tikzpicture}
               
\draw[dashed]  (0,1) -- (1,1)  (1,2) -- (2,2) -- (2,1) -- (1,1) (2,0) -- (2,1); 
 \draw[line width=.7,snake=zigzag]  (1,2) -- (1,1) -- (1,0)  (2,0) -- (0,0) -- (0,2) -- (1,2) ;
                 
\node[B] () at (2,1) {};      \node[B]() at (0,2) {};  
\node[B](w) at (2,2){};                     
\node[B]() at (1,1){};      \node[B]() at (1,2){};      
\node[B]() at (0,1){};      \node[B]() at (0,0){};   
          \node[T,label=below:$v$] () at (2,0) {};      \node[B]() at (2,0){}; 
          \node[T,label=below:$u$] (u) at (1,0){};  \node[B]() at (1,0){}; 
   \node[txt]() at (-.4,0){A\bf};    
          \node[txt]() at (2,2.4){B\bf};                                 
     	\end{tikzpicture}
\begin{tikzpicture}            
\draw[dashed]  (0,2) -- (2,2) -- (2,1) (1,0) -- (1,1); 
 \draw[line width=.7,snake=zigzag]   (0,0) -- (0,1) -- (2,1) -- (2,0)  
 (0,1) -- (0,2) (1,1) -- (1,2) (0,0) -- (2,0) ;
                 
\node[B] () at (2,1) {};      \node[B]() at (0,2) {};  
\node[B](w) at (2,2){};        \node[B]() at (2,0){};              
\node[B]() at (1,1){};      \node[B]() at (1,2){};      
\node[B]() at (0,1){};       
          \node[T,label=below:$u$] (u) at (0,0) {};        \node[B]() at (0,0){}; 
          \node[T,label=below:$v$] (v) at (1,0){};        \node[B]() at (1,0){};                 
     	\end{tikzpicture}
		       \hskip1truecm
\begin{tikzpicture}             
\draw[dashed]  (0,1) -- (2,1) -- (2,0) -- (1,0) (1,1) -- (1,2) ; 
 \draw[line width=.7,snake=zigzag]  (2,1) -- (2,2) -- (0,2)  -- (0,0) -- (1,0) -- (1,1) ;
                 
\node[B] () at (0,0) {};      \node[B]() at (0,2) {};  
\node[B](w) at (2,1){};        \node[B]() at (2,0){};              
\node[B]() at (1,1){};      \node[B]() at (1,2){};      
\node[B]() at (0,1){};      
    
          \node[T,label=below:$u$] (u) at (1,0) {};       \node[B]() at (1,0){};  
          \node[T,label=above:$v$] (v) at (2,2){};     \node[B]() at (2,2){};                      
     	\end{tikzpicture}
\begin{tikzpicture}             
\draw[dashed]  (0,2) -- (1,2) (1,0) -- (2,0) -- (2,1) -- (1,1) -- (1,0); 
 \draw[line width=.7,snake=zigzag]   (1,0) -- (0,0) -- (0,1) -- (1,1) -- (1,2) -- (2,2) -- (2,1) (0,1) -- (0,2) ;
                 
\node[B] () at (0,0) {};      \node[B]() at (0,2) {};  
\node[B](w) at (2,2){};        \node[B]() at (2,0){};              
\node[B]() at (1,1){};      \node[B]() at (1,2){};      
\node[B]() at (0,1){};           
     
          \node[T,label=below:$u$] (u) at (1,0) {};      \node[B]() at (1,0){}; 
          \node[T,label=right:$v$] (v) at (2,1){};   \node[B]() at (2,1){};              
     	\end{tikzpicture}

\end{center}
	\caption{AA-clips and AB-clips}
	\label{clips}

	\end{figure}
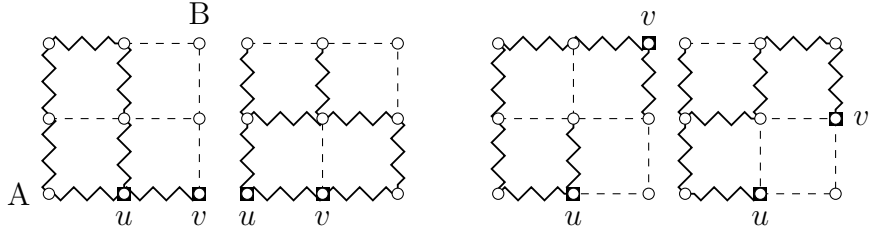
one from $x$ or $y$ to $u$, and the second one from the other terminal, $y$ or $x$, to $v$. 
A clip with $u,v \in A$ is called an AA-clip, and
a clip with $u\in A, v\in B\setminus A$  is called  an AB-clip. 
 Two-two examples of AA- and AB-clips are presented  in Fig.\ref{clips} 
 highlighted with zigzag lines.
 
 A vertex that is not a terminal (or the mate of a terminal) is called a  terminal-free vertex or simply a {\it free vertex}. A terminal 
is called a {\it singleton} if it has no pair in $Q$. 
Unlinked terminals in $L=A\cup B$ trivially escape from $Q$ without mating, unless $B\setminus A$
has two terminals which might be required as a restriction imposed on escaping. In this case a terminal must be `shifted out' from $B\setminus A$ along $L$.  If
$u\in L$ is a terminal and $v\in L$ is a free vertex, then
 vertex $u$ can be converted to a free vertex 
by {\it shifting}, denoted as $u\mapsto v$. In this way the terminal located at $u$  'moves' to the mate at location $v$ along the unique $u,v$-path $P\subset L$, and thus the edges of  $P$ are excluded from any further linkage or mating  in $Q$.
An example of shifting a terminal $s_3$ into the mate $s_3^\prime$ is seen in Fig.\ref{shift}.
  
   	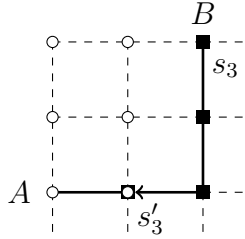
\begin{figure}[htp]
\tikzstyle{T} = [rectangle, minimum width=.1pt, fill, inner sep=2.5pt]
\tikzstyle{B} = [circle, draw=black!, minimum width=1pt, fill=white, inner sep=1.5pt]
\tikzstyle{txt}  = [circle, minimum width=1pt, draw=white, inner sep=0pt]
\tikzstyle{Wedge} = [draw,line width=.7pt,-,black!100]
\tikzstyle{M} = [circle, draw=black!, minimum width=1pt, fill=white, inner sep=1pt]
\begin{center}

  		\begin{tikzpicture}
  \foreach \x in {0,1,2}\draw[dashed](\x,0)--(\x,-.6){};   
  \foreach \y in {1,2}\draw[dashed](0,\y)--(2,\y){};      
   \foreach \x in {0,1}\draw[dashed](\x,0)--(\x,2){};   
  \foreach \y in {0,1,2}\draw[dashed](2,\y)--(2.6,\y){};  
     \draw[line width=1pt]  (0,0)--(1,0);    
   \draw[->,line width=1pt]  (2,2)--(2,0)--(1.1,0);   
         \foreach \x in {0,1,2}\foreach \y in {0,1,2}\node[B]()at(\x,\y){};                        
                  
            \node[T,label=above:$B$] (B) at (2,2) {};    
                                            
   \node[T] (s3') at (1,0) {}; \node[B] (s3') at (1,0) {};       
   \node[T] () at (2,0) {};   \node[T] () at (2,1) {};               
                   \node () at (1.3,-.35) {$s_3^\prime$};               
    \node () at (2.3,1.65) {$s_3$};  
        \node[label= left:$A$]()at(0,0){};
             
     	\end{tikzpicture}	
	
\end{center}
	\caption{Shifting $s_3$ into mate $s_3^\prime$}
	\label{shift}

	\end{figure}

\section{The escaping lemmas}
\label{proofs}
Let $Q=P_3\Box P_3$ be a $3\times 3$ grid defined as the Cartesian product of two $3$-paths. 
 Let $L=A\cup B$, where $A=A(3)$ and $B=B(3)$ are the last row and the last column of the vertex matrix
 $\{(i,j)\mid 1\leq i,j\leq 3\}$.
 Let $T\subset Q$
be a set of at most eight distinct  terminals, partitioned into at most four terminal pairs  and singletons with no pairs.

 \begin{lemma}
\label{heavy78}
Assume that $T$ contains either $4$ terminal pairs or $3$ terminal pairs plus one singleton terminal.
Then there is
 a linkage for $2$ or more pairs in $Q$, and there exist edge disjoint escape  paths for the unlinked terminals into distinct vertices of $L$
\end{lemma}
\begin{proof}  Set 
$S=Q-L$, for the $2\times 2$ square of $Q$. We have
$\|Q\|=8$ with $\pi_i\subset Q$, $1\leq i\leq 4$, or $\|Q\|=7$
with $\pi_i\subset Q$, $1\leq i\leq 3$, and $s_0\in Q$ is the singleton.

Case a:  $\|S\|=2$. If $S$ contains a pair, say $\pi_1\subset S$, then there is an $s_1,t_1$-path edge disjoint from $L$. The remaining terminals are in $A\cup B$ where there is another linkage for a second pair.
 If  none of the two terminals in $S$ is a singleton, say $s_1,s_2\in S$, then there is a linkage for $\pi_1,\pi_2$ in  $Q$ which is  $2$-path-pairable, by Lemma \ref{w2linked}.
The remaining terminals are in $L$, thus we are done. 

Assume now that $s_1,s_0\in S$. Let $H$ be the subgraph induced by $L+\cup\{(2,2)\}$, let $\overline{H}$ be the complement of $H$. Observe  that $\pi_2,\pi_3\subset L$ has a linkage $P_2,P_3\subset H$, furthermore, $\overline{H}$ extends intto a clip to escape the two
terminals in $S$  (see Fig.\ref{S2single} (i) and (ii)).

	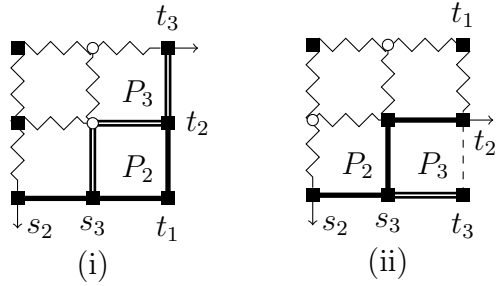
\begin{figure}[htp]
  \tikzstyle{T} = [rectangle, minimum width=.1pt, fill, inner sep=2.5pt]
\tikzstyle{B} = [circle, draw=black!, minimum width=1pt, fill=white, inner sep=1.5pt]
\tikzstyle{txt}  = [circle, minimum width=1pt, draw=white, inner sep=0pt]
\tikzstyle{Wedge} = [draw,line width=.7pt,-,black!100]
\tikzstyle{M} = [circle, draw=black!, minimum width=1pt, fill=white, inner sep=1pt]
\begin{center}
\begin{tikzpicture}     
     \draw[line width=2pt] (0,0) -- (2,0) -- (2,1);
     \draw[double,line width=1pt]  (1,0)--(1,1)--(2,1)--(2,2); 
   \draw[snake]  (0,2) -- (0,0.1);\draw[->] (0,0)--(0,-.4);
     \draw[snake]  (0,2) -- (1.9,2);   \draw[->] (2,2) -- (2.4,2);
     \draw[snake]  (0,1) -- (1,1) (1,1)--(1,2);
 
   \foreach \x in {0,1,2} \foreach \y in {0,1,2} \node[B] () at (\x,\y) {}; 
   \node[T,label=below:$t_1$] () at (2,0) {};   
    \node[T,label=right:$t_2$]() at (2,1){};
    \node[T,label=above:$t_3$] () at (2,2) {};     
    \node[T,label=below:$s_3$] () at (1,0) {};             
     \node[T] (s1) at (0,1) {};       \node[T] (s1) at (0,0) {};    
        \node[T] (s0) at (0,2) {};  
            \node[txt] () at (.3,-0.4) {$s_2$};                         
                \node[txt] () at (1.6,.4) {$P_2$};      
                \node[txt] () at (1.6,1.4) {$P_3$};   
       \node[txt]() at (1,-.9){(i)};         
     	\end{tikzpicture}
	\hskip1truecm
	\begin{tikzpicture}    
     \draw[line width=2pt]   (0,0)--(1,0)--(1,1)--(2,1); 
     \draw[double,line width=1pt]  (1,0) -- (2,0);
 
        \draw[dashed] (2,0) -- (2,1);
   \draw[snake]  (0,2) -- (0,0.1);\draw[->] (0,0)--(0,-.4);
     \draw[snake]  (0,2) -- (2,2) (2,2)--(2,1.1);   
     \draw[->] (2,1) -- (2.4,1);
     \draw[snake]  (0,1) -- (1,1) (1,1)--(1,2);
 
   \foreach \x in {0,1,2} \foreach \y in {0,1,2} \node[B] () at (\x,\y) {}; 
   \node[T,label=below:$t_3$] () at (2,0) {};   
    \node[T]() at (2,1){}; \node () at (2.3,.7) {$t_2$}; 
    \node[T,label=above:$t_1$] () at (2,2) {};     
    \node[T,label=below:$s_3$] () at (1,0) {};             
     \node[T] (s1) at (0,2) {};       \node[T] (s1) at (0,0) {};    
        \node[T] () at (0,2) {};    \node[T] () at (1,1) {};  
            \node[txt] () at (.3,-0.4) {$s_2$};                         
                \node[txt] () at (1.6,.4) {$P_3$};      
                \node[txt] () at (.6,.4) {$P_2$};  
      \node[txt]() at (1,-.9){(ii)};            
     	\end{tikzpicture}

\caption{$s_0\in S$ is a singleton}
	\label{S2single}
		\end{center}
	\end{figure}

Case b: $\|S\|=3$. If $s_1,t_1,s_2\in S$ 
then  the $2$-path-pairable $Q$ has a linkage  for $\pi_1,\pi_2$, and the remaining terminals are on $L$.
Let  $s_1,t_1,s_0\in S$ and $s_0=(i,j)$, $1\leq i,j\leq 2$. Take an $s_1,t_1$-path $P_1\subset S$ not containing $s_0$ and mate $s_0$ into $(3,j)$ along $B(j)$.  If $(3,j)$ is not a terminal or $(3,j)\in \pi_2$, then let $P_2\subset L$ be a linkage for $\pi_2$ (see Fig.\ref{S3single} (i)). 
 
 	\begin{figure}[htp]
\tikzstyle{T} = [rectangle, minimum width=.1pt, fill, inner sep=2.5pt]
\tikzstyle{B} = [circle, draw=black!, minimum width=1pt, fill=white, inner sep=1.5pt]
\tikzstyle{txt}  = [circle, minimum width=1pt, draw=white, inner sep=0pt]
\tikzstyle{Wedge} = [draw,line width=.7pt,-,black!100]
\tikzstyle{M} = [circle, draw=black!, minimum width=1pt, fill=white, inner sep=1pt]
\begin{center}

\begin{tikzpicture}
   \node[txt] (origo) at (0,-.7){};
      \draw[snake] (1,2) --  (1,0.1);\draw[->] (1,0)--(1,-.4);
      \draw[dashed] (1,0)--(0,0) -- (0,1) (0,2)-- (2,2)-- (2,1)--(1,1);
   \draw[line width=2pt] (1,0) -- (2,0) -- (2,1) (0,2)--(0,1)--(1,1);  
                 
                    \node[T]() at (2,1){};
                    \node[T] () at (2,2) {};  
                    \node[T] () at (1,1) {}; 
                   \node[B] () at (1,2) {}; 
                  \node[B] () at (0,0) {};          
                  \node[T] () at (2,0) {};    
              \node[T] () at (0,2) {};  
                  \node[B] () at (0,1) {};
              \node[T,label=above:$s_1$] (s1) at (0,2) {};              
  
        \node[T,label=above:$s_0$] (s0) at (1,2) {};
        \node[txt] () at (.7,-0.3) {$s_2$};      
         \node[T] (s2) at (1,0) {};  
         
          \node[T,label=right:$t_2$] (t2) at (2,1) {};       
       \node[txt] () at (1.25,.75) {$t_1$};    
    \node[txt] () at (.35,1.35) {$P_1$};   
    \node[txt] () at (1.65,.35) {$P_2$}; 
     
             \node[txt]() at (1,-.9){(i)};     
     	\end{tikzpicture}
\hskip.5truecm
		\begin{tikzpicture}
 \draw[dashed]   (1,1) -- (0,1) (0,2)--(2,2);
   \draw[line width=2pt]    (2,0) -- (2,1) -- (1,1)--(1,0)--(2,0);
   \draw[->,line width=1pt] (2,2) -- (2,1.1);
      \draw[snake] (0,2)--(0,0.1) (0,0)--(.9,0); 
      \draw[->] (1,0) -- (1,-.4);
        \draw[->,line width=1pt]  (1,2)--(1,1.1);   
           
              \node[T] () at (1,1) {};     \node[T] () at (2,1) {};     
              \node[T] () at (1,0) {};
              \node[T] () at (2,0) {};                
               \node[T,label=below:$s_4$] () at (0,0) {};
                  \node[B] () at (0,1) {};   
                  
            \node[T,label=above:$s_3$] (s3) at (0,2) {};                                    
            \node[T,label=above:$s_1$] (s1) at (1,2) {};           
              \node[T,label=above:$t_2$] (t2) at (2,2) {};     
                   \node[txt] () at (1.3,-.3) {$t_1$};               
               \node[txt] () at (1.25,.75) {$x$};  
                    \node[txt] () at (.7,1.3) {$s_2$}; 
                \node[txt] () at (1.65,.35) {\bf C};     
   
       \node[txt]() at (1,-.9){(ii)};  
     	\end{tikzpicture}	
		\hskip.5truecm
			\begin{tikzpicture}
		 
   \draw[dashed]   (1,1) -- (1,0);
   \draw[line width=2pt]    (2,0) -- (2,1) -- (0,1)--(0,0)--(2,0);
      \draw[->,line width=1.2pt] (2,2) -- (2,1.1); 
       \draw[->,line width=1.2pt] (1,2) -- (0,2)-- (0,1.1);
        \draw[snake] (1,1)--(1,2) --(1.9,2); \draw[->]  (2,2)--(2.4,2);   
           
              \node[T] () at (1,1) {}; \node[B] () at (0,2) {};
              \node[T] () at (1,0) {};
              \node[B] () at (2,0) {};
              \node[B] () at (0,2) {};        \node[T] () at (2,2) {};           
               \node[T] () at (0,0) {};
                
                \node[T,label=left:$s_2$] (s2) at (0,1) {};                                  
            \node[T,label=above:$s_1$] (s1) at (1,2) {};           
              \node[T,label=above:$z$] (t1) at (2,2) {};    
               \node[T,label=right:$v$] () at (2,1) {};    
                   \node[txt] () at (2.25,1.75) {$t_2$};  
               \node[txt] () at (1.25,.75) {$s_0$};  
                    \node[txt] () at (.25,.75) {$y$};  
                \node[txt] () at (.5,.35) {\bf C};     
    \node[txt] () at (.4,1.65) {$P_{12}$};       
    \node[txt] () at (1.4,1.65) {$P_{0}$}; 
        \node[txt]() at (1,-.9){(iii)};  
     	\end{tikzpicture}
		\hskip.3truecm	
		\begin{tikzpicture}
		 
   \draw[dashed]   (1,1) -- (1,0) (0,1) -- (0,2);
   \draw[line width=2pt]    (2,0) -- (2,1) -- (0,1)--(0,0)--(2,0);
     \draw[double, line width=1pt] (2,2) -- (2,1);
       \draw[->,line width=1.2pt] (0,2) -- (1,2)-- (1,1.1);
        \draw[snake] (1,2) --(1.9,2); \draw[->]  (2,2)--(2.4,2);   
           
              \node[T] () at (1,1) {}; \node[B] () at (0,1) {};
              \node[B] () at (1,0) {};
              \node[T] () at (2,0) {};
              \node[B] () at (0,2) {};        \node[T] () at (2,2) {};           
               \node[T] () at (0,0) {};
                
                \node[T,label=left:$s_2$] (s2) at (0,2) {};                                  
            \node[T,label=above:$s_0$] (s2) at (1,2) {};           
              \node[T,label=above:$z$] (z) at (2,2) {};    
               \node[T,label=right:$t_3$] (t4) at (2,1) {};    
        \node[txt] () at (2.3,1.75) {$s_3$};  
               \node[txt] () at (1.25,.75) {$s_1$};  
                    \node[txt] () at (.75,.75) {$y$};  
                \node[txt] () at (.5,.35) {\bf C};     
    \node[txt] () at (.6,1.65) {$P_{12}$};   
     \node[txt] () at (1.65,1.4) {$P_3$};   
        \node[txt]() at (1,-.9){(iv)};  
     	\end{tikzpicture}	
\caption{$\|S\|=3$}
	\label{S3single}
		\end{center}
	\end{figure}
	Let $s_1,s_2,s_3\in S$ and let $C=Q-(A(1)\cup B(1))$.  Assume that $s_1,s_2$ are the closest terminals  in $S$ to $x=(2,2)$. We define a 
  framing $[C,x]\subset Q$, by mating $s_1, s_2\in S$ into $x\in C$. Terminals $t_1,t_2\notin C$ can be shifted
  onto the cycle $C$. Observe that $[C,x]$ is edge disjoint
  from the path $A(1)\cup B(1)$, and $s_3\in A(1)\cup B(1)$. 
  Thus $s_3$ can escape either at $(1,3)$ or at $(3,1)$. If neither of these vertices is free, then $\pi_4=\{(1,3),(3,1)\}$ and hence  the  neighbors $(3,2)$ or $(2,3)$ can serve as an exit for $s_3$
  (see Fig.\ref{S3single}(ii)).
  
 Some of the previous solutions work when $s_3$ is replaced with the singleton $s_0$. In particular, for $s_0,s_1,s_2\in S$
 it is enough to consider the cases where $s_0\neq(1,1)$ (since otherwise,  $s_1,s_2$ are closest to $(2,2)$, and either singleton or not, the third terminal in $S$ escapes as above).
 
 By symmetry, we may assume that $s_0\in B(2)$, furthermore,
 if $s_0=(2,2)$, then $s_1=(2,1)$. Observe that
either $s_1$ or $s_2$ are in $C$. 
 Let $P_0\subset B(2)\cup A(1)$ be the (unique) path from $s_0$ to $z=(1,3)$,  and let $C$ be the $6$-cycle spanning $Q-A(1)$. Define $P_{12}\subset S$ to be
 the $s_1,s_2$-path  edge disjoint from $P_0\cup C$.
 Thus we obtain a framing $[C,y]$ with
 $y\in \{(2,1),(2,2)\}$ and the mating path $P_{12}$. If $z$ is a terminal, then shift it to $v=(2,3)$ thus clearing $z$ for escaping $s_0$.  If $z\in\pi_3$
 then either it can be linked to $v\in\pi_3$, or 
 $v\in \pi_j$, $j=1,2$, thus $v$ is free or it becomes free due to $P_j$  (see Fig.\ref{S3single} (iii) and (iv)).\\

Case c: $\|S\|=4$. If $S$ contains two pairs, say $\pi_1,\pi_2\subset S$, then there is a linkage for them in $Q$, and the remaining terminals are in $L$. 

Let   $s_1,t_1,s_2,s\in S$, where $s=s_0$ or $s_3$.
We take an $s_1,t_1$-path $P_1\subset S$ such that
$(1,1)$ is not an interior vertex of $P_1$.
Let $P\subset Q$ be an $s_2,s$-path containing $t_2\in L$ and
edge disjoint from $P_1$. Then the $s_2,t_2$-subpath $P_2\subset P$
is a linkage for $\pi_2$, and the $s,t_2$-subpath of $P$ is an escape path for $s$
(see Fig.\ref{2p4m} (i) and (ii)).

	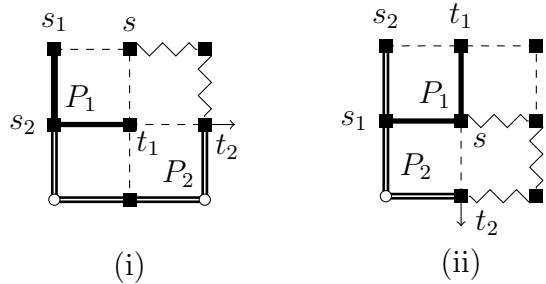
\begin{figure}[htp]
\tikzstyle{T} = [rectangle, minimum width=.1pt, fill, inner sep=2.5pt]
\tikzstyle{B} = [circle, draw=black!, minimum width=1pt, fill=white, inner sep=1.5pt]
\tikzstyle{txt}  = [circle, minimum width=1pt, draw=white, inner sep=0pt]
\tikzstyle{Wedge} = [draw,line width=.7pt,-,black!100]
\tikzstyle{M} = [circle, draw=black!, minimum width=1pt, fill=white, inner sep=1pt]
\begin{center}
\begin{tikzpicture}
      \draw[snake] (1,2) --  (1.9,2) (2,2)--(2,1.1);
      \draw[->] (2,1)--(2.4,1);
      \draw[dashed] (1,1)-- (2,1) (0,2)-- (1,2)--(1,0);
   \draw[double,line width=1pt] (0,2)--(0,0)--(2,0) --(2,1);
   \draw[line width=2pt]  (0,2)--(0,1)--(1,1);  
                 
                    \node[T]() at (2,1){};
                    \node[T] () at (2,2) {};  
                    \node[T] () at (1,1) {}; 
                   \node[B] () at (1,2) {}; 
                  \node[B] () at (0,0) {};          
                  \node[B] () at (2,0) {};    
              \node[T] () at (0,2) {};  
                  \node[T,label=left:$s_2$] () at (0,1) {};
              \node[T,label=above:$s_1$] (s1) at (0,2) {};              
  
        \node[T,label=above:$s$] (s) at (1,2) {};
        \node[txt] () at (2.3,.7) {$t_2$};      
         \node[T] (s2) at (1,0) {};  
               
       \node[txt] () at (1.25,.75) {$t_1$};    
    \node[txt] () at (.35,1.35) {$P_1$};   
    \node[txt] () at (1.65,.35) {$P_2$}; 
     
             \node[txt]() at (1,-.9){(i)};     
     	\end{tikzpicture}
\hskip1truecm
\begin{tikzpicture}
\draw[snake] (1,1) --  (1.9,1) (2,1)--(2,.1) (2,0)--(1.1,0);
      \draw[->] (1,0)--(1,-.4);
      \draw[dashed] (1,0)-- (1,1) (0,2)-- (2,2)--(2,1);
   \draw[double,line width=1pt] (0,2)--(0,0) -- (1,0);
   \draw[line width=2pt]  (0,1)--(1,1)--(1,2);  
                 
                    \node[T]() at (2,1){};
                    \node[T] () at (2,2) {};  
                    \node[T] () at (1,1) {}; 
                   \node[B] () at (1,2) {}; 
                  \node[B] () at (0,0) {};          
                  \node[T] () at (2,0) {};    
              \node[T] () at (0,2) {};  
                  \node[T,label=left:$s_1$] (s1) at (0,1) {};
              \node[T,label=above:$s_2$] (s2) at (0,2) {};              
  
        \node[T,label=above:$t_1$] (t1) at (1,2) {};    
         \node[T] (t2) at (1,0) {};  
                      
       \node[txt] () at (1.35,-.35) {$t_2$};    
       \node[txt] () at (1.25,.75) {$s$};    
    \node[txt] () at (.65,1.35) {$P_1$};   
    \node[txt] () at (.4,.4) {$P_2$}; 
     
       \node[txt]() at (1,-.9){(ii)};  
     	\end{tikzpicture}	
\caption{$\|S\|=4$, $S$ contains a pair}
	\label{2p4m}
		\end{center}
	\end{figure}
We left with the case when $s_1,s_2,s_3,s\in S$, where
$s=s_0$ or $s_4$. 

If $x=(2,2)$ is not a singleton, say $s_1=x$, then w.l.o.g. assume that $s_2=(2,1)$. We define a frame a frame $[C,x]$ for 
$\pi_1,\pi_2$ with the $4$-cycle  $C=Q-(A(1)\cup B(1))$ and by shifting $s_2$ to $x$. 
 We mate the terminal at $(1,2)$ to $z=(1,3)$, and 
 the terminal at $(1,1)$ to $v=(3,1)$, furthermore,
 we shift $z,v$ to $z^\prime,v^\prime\in C$, respectively. 
 Observe that a conflict occurs if $\{z,z^\prime\}=\{t_3,t_4\}$ which can be resolved by a linkage $P_3$ or $P_4$, since in this case $s=s_3$ and thus $(1,2)\in\{s_3,s_4\}$. The same is true for
 handling the conflict $\{v,v^\prime\}=\{t_3,t_4\}$
  (see in Fig.\ref{2p3mSfull} (i) and (ii)).

	\begin{figure}[htp]
\tikzstyle{T} = [rectangle, minimum width=.1pt, fill, inner sep=2.5pt]
\tikzstyle{B} = [circle, draw=black!, minimum width=1pt, fill=white, inner sep=1.5pt]
\tikzstyle{txt}  = [circle, minimum width=1pt, draw=white, inner sep=0pt]
\tikzstyle{Wedge} = [draw,line width=.7pt,-,black!100]
\tikzstyle{M} = [circle, draw=black!, minimum width=1pt, fill=white, inner sep=1pt]
\begin{center}
\begin{tikzpicture}
     
\draw[dashed]  (0,2) -- (1,2) -- (1,1);
   \draw[line width=2pt]  (1,1) -- (2,1) -- (2,0) -- (1,0) -- (1,1);
    \draw[snake]  (1,2) --  (1.9,2) ;\draw[->](2,2)--(2.4,2);
      \draw[->,line width=1pt] (0,0) -- (.9,0);
       \draw[snake] (0,2) -- (0,0.1);\draw[->](0,0)--(0,-.4);
       \draw[->,line width=1pt] (0,1) -- (.9,1);
  \draw[->,line width=1pt] (2,2) -- (2,1.1)  ;
 
  \node[B,label=right:$z^\prime$]() at (2,1){};   
   
  \foreach \x in {0,1,2}\foreach \y in {0,1,2} \node[B]() at (\x,\y){};     
        \node[T,label=above:$s$] (s) at (1,2) {};
        \node[T,label=left:$s_2$] (s2) at (0,1) {};
             \node[T,label=above:$s_3$] (s3) at (0,2) {};
              \node[T,label=left:$v$] (v) at (0,0) {};      
  \node[T,label=above:$z$] () at (2,2) {};
                   \node[txt] () at (1,-.25) {$v^\prime$};
        \node[T] () at (1,1) {};           \node[T] () at (2,1) {};    
\node[txt] () at (1.3,.75) {$x$};
\node[txt] () at (.75,1.25) {$s_1$};
 \node[txt]() at (1.6,.4){$C$};
  
          \node[txt]() at (1,-.9){(i)};   
     	\end{tikzpicture}
	\begin{tikzpicture}
     
\draw[dashed]  (0,2) -- (1,2) -- (1,1);
   \draw[line width=2pt]  (1,1) -- (2,1) -- (2,0) -- (1,0) -- (1,1);
      \draw[->,line width=1pt] (0,0) -- (.9,0);
       \draw[snake] (0,2) -- (0,0.1);\draw[->](0,0)--(0,-.4);
       \draw[->,line width=1pt] (0,1) -- (.9,1);
  \draw[line width=1pt] (2,2) -- (2,1.1)  ;
 \draw[double,line width=1pt] (1,2)--(2,2)--(2,1);
 
  \node[B,label=right:$t_4$]() at (2,1){};   
   
  \foreach \x in {0,1,2}\foreach \y in {0,1,2} \node[B]() at (\x,\y){};     
        \node[T,label=above:$s_4$] (s4) at (1,2) {};
        \node[T,label=left:$s_2$] (s2) at (0,1) {};
             \node[T,label=above:$s_3$] (s3) at (0,2) {};
              \node[T,label=left:$t_1$] (t1) at (0,0) {};      
  \node[T,label=above:$z$] () at (2,2) {};
                   \node[txt] () at (1,-.25) {$v^\prime$};
        \node[T] () at (1,1) {};     \node[T] () at (1,0) {};        
          \node[T] () at (2,1) {};    
\node[txt] () at (1.3,.75) {$x$};
\node[txt] () at (.75,1.25) {$s_1$};
 \node[txt]() at (1.6,.4){$C$};
   \node[txt]() at (1.6,1.6){$P_4$};
          \node[txt]() at (1,-.9){(ii)};   
     	\end{tikzpicture}
\begin{tikzpicture}
   
   \draw[line width=2pt]  (2,2) -- (2,0)--(1,0)-- (1,2) -- (2,2) ;
    \draw[->,line width=1pt]  (0,2) --  (.9,2) ;   
  \draw[snake] (0,1) -- (0,0.1);\draw[->] (0,0)--(0,-.4);
  \draw[->,line width=1pt] (0,0)--(0.9,0);
    \draw[snake] (1,1) -- (1.9,1);\draw[->] (2,1)--(2.4,1);     
   \draw[dashed]  (0,2) -- (0,1) -- (1,1) ;

  \foreach \x in {0,1,2}\foreach \y in {0,1,2} \node[B]() at (\x,\y){};      
      \node[T]() at (1,1){};  
            \node[T]() at (1,0){};              
                  \node[T]() at (2,1){};               
         \node[txt] () at (2.3,.75) {$t_1$};
        \node[T,label=above:$s_2$] (s2) at (1,2) {};
        \node[T,label=left:$s_3$] (s3) at (0,1) {};
             \node[T,label=above:$s_1$] (s1) at (0,2) {};
         \node[T]()at(0,0){};                 
\node[txt] () at (1.3,1.75) {$y$};
\node[txt] () at (.75,1.255) {$s_0$};
     \node[txt]() at (1.5,.4){$C$};
  
         \node[txt]() at (1,-.9){(iii)};   
     	\end{tikzpicture}
	\begin{tikzpicture}
   
   \draw[line width=2pt]  (0,1) -- (0,0)--(2,0)-- (2,1) ;
    \draw[snake]  (0,2) --  (1.9,2) ;   \draw[->](2,2)--(2.4,2);
    \draw[snake] (1,1) -- (1.9,1);\draw[->] (2,1)--(2.4,1);     
  \draw[dashed]  (0,2) -- (0,1) -- (1,1)  (2,2)--(2,1);
\draw[double,line width=1pt](1,2)--(1,0);
  \foreach \x in {0,1,2}\foreach \y in {0,1,2} \node[B]() at (\x,\y){};      
      \node[T]() at (1,1){};  
            \node[T,label=below:$t_2$]() at (1,0){};              
                  \node[T]() at (2,1){};                       
        \node[T,label=above:$s_2$] (s2) at (1,2) {};
        \node[T,label=left:$s_3$] (s3) at (0,1) {};
             \node[T,label=above:$s_1$] (s1) at (0,2) {};
         \node[T,label=below:$t_1$]()at(0,0){};                 
\node[txt] () at (2.3,.75) {$t_3$};
\node[txt] () at (.75,1.255) {$s_0$};
     \node[txt]() at (1.6,.4){$P_3$};
     \node[txt]() at (.6,.6){$P_2$}; 
         \node[txt]() at (1,-.9){(iv)};   
     	\end{tikzpicture}
\caption{$\|S\|=4$, $S$ contains no pair}
	\label{2p3mSfull}
		\end{center}
	\end{figure}
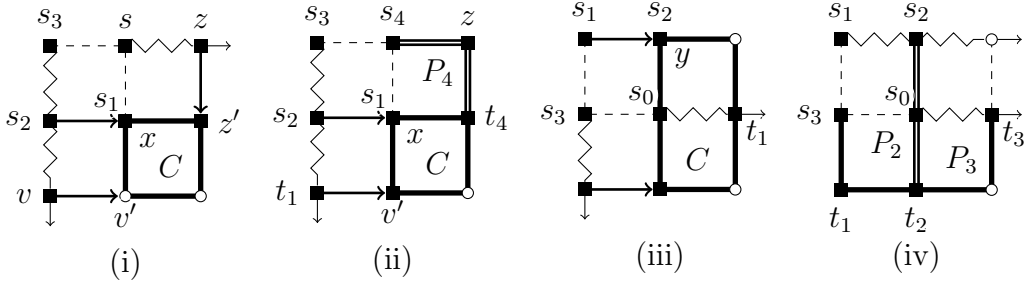

Finally we assume that $s=s_0=(2,2)$. 
Let $s_1=(1,1)$, $s_2=(1,2)$,  $s_3=(2,1)$, and define  a frame $[C,y]$
  for $\pi_1,\pi_2$ with the $6$-cycle on the vertices of
 $Q-B(1)$ and $y=(1,2)$. Then $s_3$ and $s_0$ can escape at
 $(3,1)$ and $(2,3)$, respectively, if $(2,3)\neq t_3$ (see Fig.\ref{2p3mSfull} (iii)). By diagonal symmetry (swapping rows and columns), we may assume that there is a solution if $t_2\neq (3,2)$. Now let $t_2=(3,2)$, $t_3=(2,3)$, and since
 one of the corners $(1,3)$ and $(3,1)$ is terminal free, assume that  $(1,3)\neq t_1$.
 For this case a direct solution is shown in  Fig.\ref{2p3mSfull} (iv).
 \end{proof}

   \begin{lemma}
   \label{heavy6}
If $T$ has $6$ terminals, then there is
 a linkage for one or more pairs in $Q$, and there exist edge disjoint escape paths for the unlinked terminals into distinct exit vertices of $L$
 such that  $B\setminus A$ contains at most one exit.
    \end{lemma}
   \begin{proof}
 We assume w.l.o.g. that $Q$ contains the pairs $\pi_1,\pi_2$, and the other two terminals in $Q$ are singletons, say $s_3$ and $s_4$.  Set $S=Q-L$. We will use clips and shifting 
along $L$ as defined in Section \ref{tools}. \\

  Case a:  $\|S\|=1$. 
  We may assume that $\pi_1\subset L$ and
$s_\ell\in S$ ($2\leq \ell\leq 4$). The linkage $P_1$ in $L$ for $\pi_1$ results in  two free end vertices  $s_1$ and $t_1$ available as exits for escaping. 
Then take a path from $s_\ell$ leading to 
 $(3,1)$ and extend it along $L$  until  the first end vertex of $P_1$ is reached.
If $\pi_1\notin A$, then  $B\setminus A$ contains at most one exit. If  $\pi_1\subset A$
 then  $(1,3),(2,3)$ are both terminals.  Then one of them should escape to exit  at $A$; this can be done by shifting the terminal along  $L$ to the other free end vertex of $P_1$. \\

Case b: $\pi_1\subset S$. There is at least one free vertex  $w\in L$.
If $\|S\|=2$, then we take a
  linkage $P_1$ in $S$ for $\pi_1$; furthermore, in case of $w\notin B\setminus A$
one terminal in $B\setminus A$ will be shifted into $w$ along $L$ to escape there. 

If we have $\pi_2\subset L$, then the linkage $P_2$  in $L$ for $\pi_2$ yields two more free vertices  $s_2,t_2\in L$. In case of $\pi_2\subset A$, let  $t_2$ be the closest endvertex of $P_2$ to $B\setminus A$,
 then the shift $(2,3)\mapsto t_2$ makes $(2,3)$ free.
Thus we may assume that $\pi_2\not\subset L$, w.l.o.g. let
$s_1,t_1,s_2\in S$. 

If $\|S\|=3$, then $L$ contains two free vertices $w_1,w_2\in L$. If $u=(3,1)$ is a terminal  we make $u$ free by the shifting 
$u\mapsto w_1$, where $w_1$ is the closest free vertex to $u$. 
Now there is an escape path in the subgrid $Q-B$ from $s_2$ to $u$, and an edge disjoint linkage $P_1$ in $S$ for $\pi_1$.
If $\|B\setminus A\|=2$, then $w_2\in A$, and the shift $(2,3)\mapsto w_2$ along $L$ makes $(2,3)$ free.

For $\|S\|=4$,  $S$ contains the pair $\pi_1=\{s_1,t_1\}$, and two more terminals  $p,q\in\{s_2, t_2, s_3,s_4\}$.
If $\|B\setminus A\|=2$, then $u=(3,1),v=(3,2)$ and $x=(3,3)$ are free vertices.
First we shift $(2,3)\mapsto x$, then
take a linkage $P_1$  in $S$ for $\pi_1$. Terminals $p,q$ escape $Q$ through free vertices $u,v$. For this purpose we define an AA-clip $CL(u,v)$
not using edges of $P_1$
as shown in Fig.\ref{s=4}. The clip in picture (i) 
is for the cases when $(1,1)\notin\pi_1$ 
; the clip in (ii)  is defined for $(1,1)\in \pi_1$,
finally, picture (ii) shows the mating for the remaining case $\pi_1=\{(1,1),(2,2)\}$. 

\begin{figure}[htp]
 \tikzstyle{T} = [rectangle, minimum width=.1pt, fill, inner sep=2.5pt]
\tikzstyle{B} = [circle, draw=black!, minimum width=1pt, fill=white, inner sep=1.5pt]
\tikzstyle{M} = [circle, draw=black!, minimum width=1pt, fill=white, inner sep=1pt]
\begin{center}
\begin{tikzpicture}
           
\draw[dashed]  (1,1) -- (3,1); 
 \draw[snake]  (3,3) -- (1,3) -- (1,1.1) (3,3) -- (3,2)  --(2,2)-- (2,1.1);
  \draw[->]  (1,1) -- (1,.6); \draw[->]  (2,1) -- (2,.6);
   \draw[->]  (3,1) -- (3,.6);
        \draw[->]  (3,3) -- (3.4,3);
         
   \draw[double] (2,3) -- (2,2) -- (1,2);  
\draw[->,line width=1pt] (3,2) -- (3,1.1);    
        
\foreach \x in {1,2,3}\foreach \y in{2,3} \node[T]() at (\x,\y){};                                
\foreach \x in{1,2} \node[B]() at (\x,1){};
          
\node[B,label=right:$x$]() at (3,1){}; 
\node[T]() at (3,2){}; 
\node[T,label=above:$p$](p) at (1,3){};            
     
          \node() at (2.3,.8){$v$};         \node() at (.7,.8){$u$};                           
            \node() at (1.65,2.35){$\pi_1$};     
      \node() at (2,0){(i)};     
                                
     	\end{tikzpicture}
	\hskip.5cm	
	\begin{tikzpicture}
           
\draw[dashed]  (1,1) -- (3,1) (2,3)--(3,3)--(3,2)--(2,2); 
 \draw[snake]  (2,3) -- (2,1.1)  (2,2)  --(1,2)-- (1,1.1);
  \draw[->]  (1,1) -- (1,.6); \draw[->]  (2,1) -- (2,.6);
     \draw[->]  (3,1) -- (3,.6);
       \draw[->]  (3,3) -- (3.4,3);
        
   \draw[double] (2,3) -- (1,3) -- (1,2);  
       \draw (3,2) -- (3,1);    
        
\foreach \x in {1,2,3}\foreach \y in{2,3} \node[T]() at (\x,\y){};                                
\foreach \x in{1,2} \node[B]() at (\x,1){};
 \node[B]() at (3,1){}; \node[T]() at (3,2){}; 
          
          \node[T,label=above:$s_1$](s1) at (1,3){};  

            \node() at (2.3,1.8){$p$};       
     
          \node() at (2.3,.8){$v$};         \node() at (.7,.8){$u$};                           
            \node() at (1.35,2.65){$\pi_1$};     
      \node() at (2,0){(ii)};     
                                
     	\end{tikzpicture}
	\hskip1cm		
	\begin{tikzpicture}
           
\draw[dashed]  (1,1) -- (3,1) (1,3)--(2,3)--(3,3)--(3,2)--(2,2); 
 \draw[snake]  (2,3) -- (2,1.1)  (1,2)-- (1,1.1);
  \draw[->]  (1,1) -- (1,.6); \draw[->]  (2,1) -- (2,.6);
     \draw[->]  (3,1) -- (3,.6);
      \draw[->]  (3,3) -- (3.4,3);
  
 \draw[line width=2pt] (1,3) -- (1,2) -- (2,2);  
\draw (3,2) -- (3,1);   
        
\foreach \x in {1,2,3}\foreach \y in{2,3} \node[T]() at (\x,\y){};                                
\foreach \x in{1,2} \node[B]() at (\x,1){};
     \node[B]() at (3,1){}; \node[T]() at (3,2){};      
          \node[T,label=above:$s_1$](s1) at (1,3){};  

            \node() at (2.3,1.8){$t_1$};       
     \node[T,label=left:$p$](p) at (1,2){}; 
     \node[T,label=above:$q$](q) at (2,3){};   
     
          \node() at (2.3,.8){$v$};         \node() at (.7,.8){$u$};                           
            \node() at (1.35,2.35){$P_1$};     
      \node() at (2,0){(iii)};     
                                
     	\end{tikzpicture}
		\hskip1cm		
\begin{tikzpicture}
           
\draw[->,line width=1pt]  (1,1) -- (2.9,1); 
 \draw[snake]  (3,3) -- (1,3) -- (1,1.1) (3,3) -- (3,2)  --(2,2)-- (2,1.1);
  \draw[->]  (1,1) -- (1,.6); \draw[->]  (2,1) -- (2,.6);
     \draw[->]  (3,1) -- (3,.6);
    \draw[->]  (3,2) -- (3.4,2);
    
   \draw[double] (2,3) -- (2,2) -- (1,2);  
\draw[dashed] (3,2) -- (3,1);   
        
\foreach \x in {1,2}\foreach \y in{2,3} \node[T]() at (\x,\y){};                                
\foreach \x in{1,2} \node[B]() at (\x,1){};
          \node[B,label=right:$x$]() at (3,1){};           
          \node[T]() at (1,1){};
          \node[B]() at (3,3){};\node[T]() at (3,2){};

\node[T,label=above:$p$](p) at (1,3){};            
     
          \node() at (2.3,.8){$v$};         \node() at (.7,.8){$u$};                           
            \node() at (1.65,2.35){$\pi_1$};     
      \node() at (2,0){(iv)};     
                                
     	\end{tikzpicture}
	\end{center}
	\caption{AA-clips for $\pi_1\subset S$ and $\|S\|=4$}
	\label{s=4}

	\end{figure}
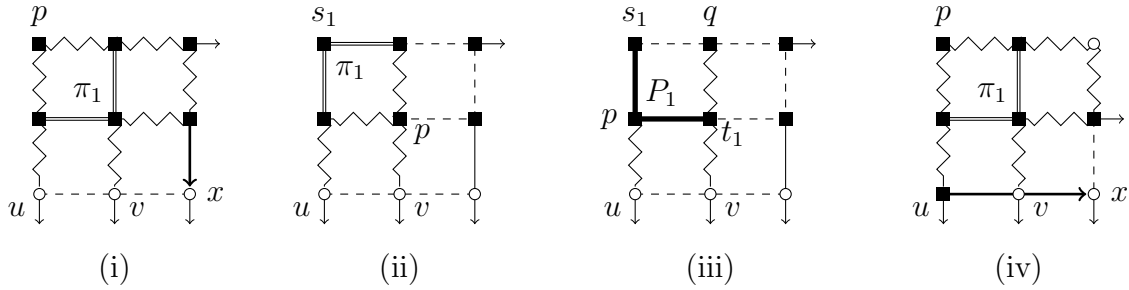

The solutions above can be applied for  $\|B\setminus A\|= 1$. In this case
there is just one terminal in $A$, and if $u$ or $v$ is a terminal, shift it to the free vertex $x$.
As an example, the AA-clip 
 in picture (i) of Fig.\ref{s=4} is repeated in picture (iv). 
  
For $\|B\setminus A\|=0$ there is a free vertex $w\in A$, if $u$ is a terminal, then make it free by the shift $u\mapsto w$.
After taking a linkage $P_1\subset S$ for $\pi_1$ we define AB-clips $CL(u,z)$, where $z=(1,3)$,
as shown in Fig.\ref{s=4AB}. The AB-clip in picture (i) 
is for $p=(1,1)$, the clip in (ii) serves for the cases when $(1,1)\in \pi_1$ and $(2,2)\not\in \pi_1$, picture (iii) is a direct solution
for $\pi_1=\{(1,1),(2,2)\}$.

	\begin{figure}[htp]
 \tikzstyle{T} = [rectangle, minimum width=.1pt, fill, inner sep=2.5pt]
\tikzstyle{B} = [circle, draw=black!, minimum width=1pt, fill=white, inner sep=1.5pt]
\tikzstyle{M} = [circle, draw=black!, minimum width=1pt, fill=white, inner sep=1pt]
\begin{center}
\begin{tikzpicture}
           
\draw[dashed]  (2,1) --(2,2)  (3,1)--(3,2); 
 \draw[snake]  (1,3) -- (1,1.1) (1,3) -- (3,3)  (2,2)--(3,2)--(3,2.9);
  \draw[->]  (1,1) -- (1,.6); \draw[->]  (3,3) -- (3.4,3);
   \draw[->]  (2,1) -- (2,.6);  \draw[->]  (3,1) -- (3,.6);
   \draw[->,line width=1pt]  (1,1)--(2.9,1);
   \draw[double] (2,3) -- (2,2) -- (1,2);  
        
\foreach \x in {1,2}\foreach \y in{2,3} \node[T]() at (\x,\y){};                                
  \node[T]() at (1,1){};  \node[T]() at (2,1){};
           \node[B]() at (3,2){};   
           
\node[B,label=right:$w$]() at (3,1){}; 
\node[T,label=above:$p$](p) at (1,3){};            
  \node[B,label=above:$z$](z) at (3,3){};   
     
                \node() at (.7,.8){$u$};                           
            \node() at (1.65,2.35){$\pi_1$};     
      \node() at (2,0){(i)};     
                                
     	\end{tikzpicture}
	\hskip1cm	
	\begin{tikzpicture}
           
\draw[dashed]  (3,2)--(2,2) -- (2,1)-- (3,1)--(3,3)  ; 
 \draw[snake]  (2,2)--(2,3) -- (2.9,3) (2,2)--(1,2)--(1,1.1);
  \draw[->]  (1,1) -- (1,.6); \draw[->]   (3,3)  --(3.4,3);
    \draw[->]  (2,1) -- (2,.6);   \draw[->]  (3,1) -- (3,.6);
  \draw[->,line width=1pt]  (1,1) -- (1.9,1);
   \draw[double] (2,3) -- (1,3) -- (1,2);  
        
\foreach \x in {1,2}\foreach \y in{2,3} \node[T]() at (\x,\y){};                                

 \node[T]() at (1,1){};  
\node[B]() at (3,2){};  \node[T]() at (3,1){};  
   \node[B,label=above:$z$](z) at (3,3){};            
    \node[T,label=above:$s_1$](s1) at (1,3){};  

   \node() at (2.3,1.8){$p$};       
   \node[B]() at (2,1){};          
         \node() at (.7,.8){$u$};         \node() at (2.3,.8){$w$};                          
            \node() at (1.35,2.65){$\pi_1$};     
      \node() at (2,0){(ii)};     
                                
     	\end{tikzpicture}
	\hskip1cm		
	\begin{tikzpicture}
           
\draw[dashed]  (1,1) -- (3,1)--(3,2)--(2,2) --(2,3)--(1,3) (3,2)--(3,3) (2,1) --(2,2); 
 \draw[snake]  (2,3) -- (2.9,3)  (1,2)-- (1,1.1);
  \draw[->]  (1,1) -- (1,.6); \draw[->]  (3,3) -- (3.4,3);
     \draw[->]  (2,1) -- (2,.6);   \draw[->]  (3,1) -- (3,.6);
 \draw[line width=2pt] (1,3) -- (1,2) -- (2,2);   
        
\foreach \x in {1,2}\foreach \y in{2,3} \node[T]() at (\x,\y){};                                
 \node[B]() at (1,1){};
          
          \node[T,label=above:$s_1$](s1) at (1,3){};  
\node[T]() at (3,1){}; \node[T]() at (2,1){}; \node[B]() at (3,2){}; 
            \node() at (2.3,1.8){$t_1$};       
     \node[T,label=left:$p$](p) at (1,2){}; 
     \node[T,label=above:$q$](q) at (2,3){};   
    \node[B,label=above:$z$](z) at (3,3){};     
        
           \node() at (.7,.8){$u$};                           
            \node() at (1.35,2.35){$P_1$};     
      \node() at (2,0){(iii)};     
                                
     	\end{tikzpicture}

	\end{center}
	\caption{AB-clips for $\pi_1\subset S$ and $\|S\|=4$}
	\label{s=4AB}

	\end{figure}
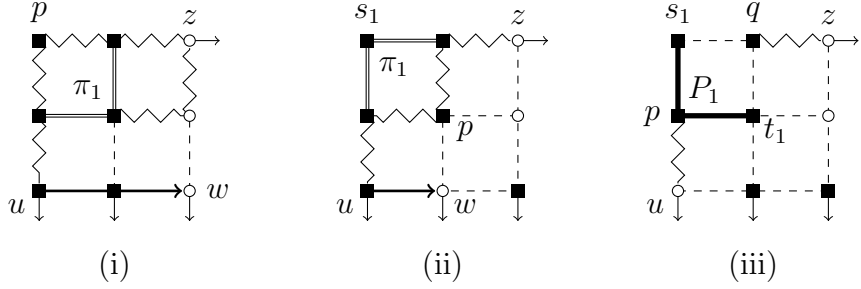

Case c:  $\pi_1\subset A$.  Due to Case a, we may assume that $\pi_2\not\subset S$, in particular, $\|S\|=2$ or $3$. For $\|S\|=2$, there is one free vertex in $L$, say $w$. Let $u=(3,1)$, $v=(3,2)$. 
The (unique) path $P_1\subset A$ is a linkage for $\pi_1$ and makes $s_1,t_1$ free. For $w\in A$, every vertex
becomes free, thus  we can move out $(2,3)$ from $B\setminus A$ by the shift
$(2,3)\mapsto (3,3)$. The (unique) $u,v$-path in
$B(1)\cup A(1)\cup B(2)$ contains $S$ and makes an AA-clip $CL(u,v)$. For $w\in B\setminus A$ no shift is necessary and an $s_1,t_1$-path in $Q$ including $S$ and edge disjoint from $P_1$ defines an AA-clip.

For $\|S\|=3$, $L$ has three terminals and 
two free vertices,
let $s_2,s_3,s_4\in S$, $s_1,t_1,t_2\in L$. 
If $t_2\in B$, then we take a linkage $P_2$ for $\pi_2$ using column $B(3)$ and the row $A(1)$ or $A(2)$ containing
$s_2$. Taking a linkage $P_1\subset A$ for $\pi_1$, vertices  $u$ and $v$ become free, thus an AA-clip $CL(u,v)$ is obtained by taking the union of the paths $B(1)$ and $B(2)$
together with an edge between them not used by $P_2$.
\begin{figure}[htp]
 \tikzstyle{T} = [rectangle, minimum width=.1pt, fill, inner sep=2.5pt]
\tikzstyle{B} = [circle, draw=black!, minimum width=1pt, fill=white, inner sep=1.5pt]
\tikzstyle{M} = [circle, draw=black!, minimum width=1pt, fill=white, inner sep=1pt]
\begin{center}
\begin{tikzpicture}
           
\draw[dashed]  (3,1) --(3,3)  (3,2)--(2,2); 
 \draw[snake]  (1,3)--(2.9,3) (1,3) -- (1,1.1)  (2,2)--(1,2);
 \draw[->]  (1,1) -- (1,.6); \draw[->]  (3,3) -- (3.4,3);
  
   \draw[line width=2pt] (1,1)--(2,1)--(2,3) ;  
   \draw[line width=2pt] (2,1) -- (3,1);
                              
  \node[T]() at (1,1){};  
   \node[T,label=below:$t_1$]() at (3,1){}; 
  \node[T,label=below:$s_1$](s1) at (2,1){};  

       \node[B]() at (3,2){};       \node[B]() at (3,3){};  \node[B]() at (1,2){};    
       \node[B]() at (2,2){};       \node[B]() at (1,3){};       
\node[T,label=above:$s_2$](s2) at (2,3){};   \node() at (.7,.8){$t_2$};     
 \node[B,label=above:$z$](z) at (3,3){};   
                                   
            \node() at (1.65,1.35){$P_2$};     
           \node() at (2.6,1.35){$P_1$};                      
     	\end{tikzpicture}	
	\hskip.5truecm
\begin{tikzpicture}
           
\draw[dashed]  (3,1) --(3,3)  (3,2)--(2,2) (1,3)--(2,3); 
 \draw[snake]  (1,2)--(2,2)--(2,1) -- (1,1.1)  (2,2)--(2,3)--(2.9,3);
 \draw[->]  (1,1) -- (1,.6); \draw[->]  (3,3) -- (3.4,3);
  
   \draw[line width=2pt] (1,1)--(1,3) ;  
   \draw[line width=2pt] (2,1) -- (3,1);
         
\foreach \y in{3} \node[T]() at (1,\y){};  
                              
  \node[T]() at (1,1){}; 
 
  \node[T,label=below:$s_1$](s1) at (2,1){}; 
  \node[T,label=below:$t_1$]() at (3,1){}; 
       \node[B]() at (3,2){};       \node[B]() at (3,3){};  \node[B]() at (1,2){};   
       \node[B]() at (2,2){};       \node[B]() at (2,3){};       
\node[T,label=above:$s_2$](s2) at (1,3){};   \node() at (.7,.8){$t_2$};     
 \node[B,label=above:$z$](z) at (3,3){};   
                                   
            \node() at (1.35,2.5){$P_2$};     
            \node() at (2.6,1.35){$P_1$};                       
     	\end{tikzpicture}	
	\hskip.5truecm
\begin{tikzpicture}
           
\draw[dashed]  (3,1) --(3,2) (1,2)--(1,3); 
 \draw[snake]  (2,2)--(1,2)--(1,1.1) (2,2)--(3,2) -- (3,2.9)  (2,3)--(2.9,3);
 \draw[->]  (1,1) -- (1,.6); \draw[->]  (3,3) -- (3.4,3);
  
   \draw[line width=2pt] (2,1)--(2,3)--(1,3) ;  
   \draw[line width=2pt] (1,1) -- (3,1);
         
\foreach \y in{3} \node[T]() at (1,\y){};  
                              
  \node[T]() at (1,1){};  
  \node[T,label=below:$t_1$]() at (3,1){}; 
  \node[T,label=below:$t_2$](t2) at (2,1){}; 

       \node[B]() at (3,2){};       \node[B]() at (3,3){};  \node[B]() at (1,2){};   
       \node[B]() at (2,2){};       \node[B]() at (2,3){};       
\node[T,label=above:$s_2$](s2) at (1,3){};   \node() at (.7,.8){$s_1$};     
 \node[B,label=above:$z$](z) at (3,3){};   
                                   
            \node() at (1.65,2.65){$P_2$};     
            \node() at (2.45,1.35){$P_1$};                       
     	\end{tikzpicture}

\end{center}
	\caption{$\pi_1\subset A$ and $\|S\|=3$}
	\label{pinA}

	\end{figure}
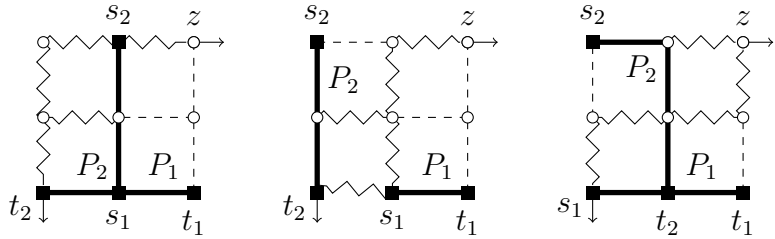

Now we may assume that $t_2\in A\setminus B$, in particular, $B\setminus A$ has no terminal. 
For the different locations of $s_2$ and $t_2$ each solution consists of a linkage for $\pi_1$, 
$\pi_2$, and an AB-clip $CL(u,z)$,
where  $u=(3,1)$ and $z=(1,3)$ (see the examples in Fig.\ref{pinA}). \\

Case d:   $\pi_1\subset B$. We have $\|S\|=2$ or $3$, by Case a. For $\|S\|=2$, let $w$ be the only free vertex in $L$
and take the (unique) path $P_1\subset B$ to be the linkage for $\pi_1$.
If all vertices of $B$ become free, we define an AB-clip $CL((3,3),(1,3))$ 
to escape the two terminals of $S$ into $L$ 
as on the left of Fig.\ref{pinB}. The same clip works for $\pi_1=B\setminus A$, thus we may assume that $t_1=(3,3)$ and $w\notin B$. For this case an AA-clip $CL(w,t_1)$  is shown on the right of Fig.\ref{pinB}.
\begin{figure}[htp]
  \tikzstyle{T} = [rectangle, minimum width=.1pt, fill, inner sep=2.5pt]
\tikzstyle{B} = [circle, draw=black!, minimum width=1pt, fill=white, inner sep=1.5pt]
\tikzstyle{txt}  = [circle, minimum width=1pt, draw=white, inner sep=0pt]
\tikzstyle{Wedge} = [draw,line width=1.5pt,-,black!100]
\tikzstyle{M} = [circle, draw=black!, minimum width=1pt, fill=white, inner sep=1pt]
\begin{center}
	\begin{tikzpicture}
           
\draw[dashed]  (2,1) --(2,3) (2,2)--(3,2); 
 \draw[snake]  (2,2)--(1,2) (1,2)--(1,2.9) (1,3)--(2.9,3) (1,2)--(1,1.1) (1,1)--(2.9,1);
 \draw[->]  (3,3) -- (3.4,3); \draw[->]  (3,1) -- (3,.6);
  
   \draw[line width=2pt] (3,1)--(3,3) ;  
         
\foreach \y in{1,3} \node[T]() at (3,\y){};  
   \foreach \x in{1,2}  \foreach \y in{1,2,3} \node[B]() at (\x,\y){};      
                         
 \node[T,label=right:$t_1$]() at (3,1){};  
  \node[B](w) at (2,1){};          
\node[T,label=right:$$]() at (3,3){};    
\node[T]() at (1,1){};   \node[T]() at (2,1){};  
          \node[B,label=right:$w$](w) at (3,2){};                            
            \node (s1) at (3.25,2.75){$s_1$};   
            \node() at (2.65,1.55){$P_1$};                       
     	\end{tikzpicture}
	\hskip1truecm
	\begin{tikzpicture}
           
\draw[dashed]  (2,3) --(3,3) (1,2)--(3,2); 
 \draw[snake]  (1,3)--(1,1.1) (1,1)--(1.9,1) (1,3)--(2,3)--(2,1.1) (2,1)--(2.9,1);
 \draw[->]  (2,1) -- (2,.6); \draw[->]  (3,1) -- (3,.6);
  
   \draw[line width=2pt] (3,1)--(3,3) ;  
         
\foreach \y in{1,2,3} \node[T]() at (3,\y){};  
   \foreach \x in{1,2}  \foreach \y in{1,2,3} \node[B]() at (\x,\y){};      
                         
 \node[T,label=right:$t_1$]() at (3,1){};  
  \node[B](w) at (2,1){};         \node() at (1.7,.7){$w$};  
\node[T,label=right:$s_1$](s1) at (3,3){};    \node[T](w) at (1,1){};  
                                   
            \node() at (2.65,2.45){$P_1$};                       
     	\end{tikzpicture}

\end{center}
	\caption{$\pi_1\subset B$ and $\|S\|=2$}
	\label{pinB}

	\end{figure}
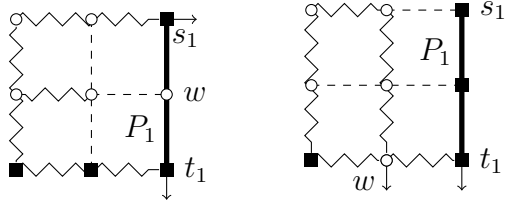

For $\|S\|=3$, 
let $s_2,s_3,s_4\in S$ and $s_1,t_1,t_2\in L$.  Each of he pictures in Fig.\ref{inB} defines two edge disjoint subgraphs of $Q$, one is an AB-clip $CL(u,z)$, where
$u\in \{(3,1),(3,2)\}$ and $z=(1,3)$, the other subgraph contains the edge disjoint paths $P_1\subseteq B$
and $P_2$ linking $\pi_1$ and $\pi_2$.

\begin{figure}[htp]
  \tikzstyle{T} = [rectangle, minimum width=.1pt, fill, inner sep=2.5pt]
\tikzstyle{B} = [circle, draw=black!, minimum width=1pt, fill=white, inner sep=1.5pt]
\tikzstyle{txt}  = [circle, minimum width=1pt, draw=white, inner sep=0pt]
\tikzstyle{Wedge} = [draw,line width=1.5pt,-,black!100]
\tikzstyle{M} = [circle, draw=black!, minimum width=1pt, fill=white, inner sep=1pt]
\begin{center}

		\begin{tikzpicture}
           
\draw[dashed]  (2,1) --(3,1) (2,2)--(3,2); 
 \draw[snake]  (1,2)--(2,2)  (1,3)--(2.9,3) (2,3)--(2,1.1) ;
 \draw[->]  (3,3) -- (3.4,3); \draw[->]  (2,1) -- (2,.6);
  
   \draw[line width=2pt] (3,1)--(3,3) (2,1)--(1,1)--(1,3);  
         
\foreach \y in{1,3} \node[T]() at (3,\y){};  
   \foreach \x in{1,2}  \foreach \y in{1,2,3} \node[B]() at (\x,\y){};      
                         
 \node[T,label=right:$$]() at (3,1){};  
  \node[T]() at (2,1){};     \node[T]() at (1,3){};  
  \node() at (2.3,.7){$u$};   
   \node[B]() at (1,1){};   
 \node[T,label=above:$z$]() at (3,3){};  
          \node[B,label=right:$$](w) at (3,2){};                            
            
            \node() at (3.35,1.45){$P_1$};     \node() at (1.35,1.5){$P_2$};         
             \node() at (2,0.2){(i)};            
     	\end{tikzpicture}
		\hskip.5truecm
		\begin{tikzpicture}
           
\draw[dashed]  (2,1) --(3,1) (2,2)--(3,2); 
 \draw[snake]  (2,2)--(1,2) (1,2)--(1,2.9) (1,3)--(2.9,3) (1,2)--(1,1.1) ;
 \draw[->]  (3,3) -- (3.4,3); \draw[->]  (1,1) -- (1,.6);
  
   \draw[line width=2pt] (3,1)--(3,3) (1,1)--(2,1)--(2,3);  
         
\foreach \y in{1,3} \node[T]() at (3,\y){};  
   \foreach \x in{1,2}  \foreach \y in{1,2,3} \node[B]() at (\x,\y){};      
                         
 \node[T,label=right:$$]() at (3,1){};  
  \node[B](w) at (2,1){};     \node[T]() at (2,3){};      
     \node[T]() at (1,1){};   \node[B]() at (2,1){};  
 \node[T,label=above:$z$]() at (3,3){};  
          \node[B,label=right:$$](w) at (3,2){};  
           \node[T,label=left:$u$](u) at (1,1){};       
                       \node() at (3.35,1.45){$P_1$};     \node() at (1.65,1.5){$P_2$};         
             \node() at (2,0.2){(ii)};            
     	\end{tikzpicture}
	\hskip1truecm	
		\begin{tikzpicture}
\draw[dashed] (2,1)--(3,1) ; 
 \draw[snake]  (1,3)--(1.9,3) (1,2)--(1,1.1) (1,1)--(1.9,1)  (2,3)--(2.9,3) (2,3)--(2,1.1) ;
 \draw[->]  (3,3) -- (3.4,3); \draw[->]  (2,1) -- (2,.6);
  
   \draw[line width=2pt] (3,1)--(3,3) (3,2)--(1,2)--(1,3);  
         
\foreach \y in{1,3} \node[T]() at (3,\y){};  
   \foreach \x in{1,2}  \foreach \y in{1,2,3} \node[B]() at (\x,\y){};      
                         
 \node[T]() at (3,1){};  \node[T]() at (1,3){};  
   \node[B]() at (2,1){};             
  \node[B]() at (1,1){};   \node[B]() at (2,1){};  
 \node[T,label=above:$z$]() at (3,3){};  
          \node[T]() at (3,2){};                  \node() at (2.3,.7){$u$};            
            
            \node() at (3.35,1.45){$P_1$};     \node() at (1.45,2.45){$P_2$};         
             \node() at (2,0.2){(iii)};            
     	\end{tikzpicture}
	\hskip.5truecm	
\begin{tikzpicture}
           
\draw[dashed]  (1,1) --(3,1) (2,2)--(2,1); 
 \draw[snake]  (2,2)--(1,2) (1,2)--(1,2.9) (1,3)--(2.9,3) (1,2)--(1,1.1) ;
 \draw[->]  (3,3) -- (3.4,3); \draw[->]  (1,1) -- (1,.6);
  
   \draw[line width=2pt] (3,1)--(3,3) (3,2)--(2,2)--(2,3);  
         
\foreach \y in{1,3} \node[T]() at (3,\y){};  
   \foreach \x in{1,2}  \foreach \y in{1,2,3} \node[B]() at (\x,\y){};      
                         
 \node[T,label=right:$$]() at (3,1){};  \node[T]() at (2,3){};  
  \node[B](w) at (2,1){};        
      \node[B,label=left:$u$](u) at (1,1){};      
     \node[B]() at (2,1){};  
 \node[T,label=above:$z$]() at (3,3){};    \node[T]() at (3,2){};  
                                             
            \node() at (3.35,1.45){$P_1$};     \node() at (2.45,2.45){$P_2$};         
             \node() at (2,0.2){(iv)};            
     	\end{tikzpicture}
	   
\end{center}
	\caption{$\pi_1\subset B$ and $\|S\|=3$}
	\label{inB}

	\end{figure}
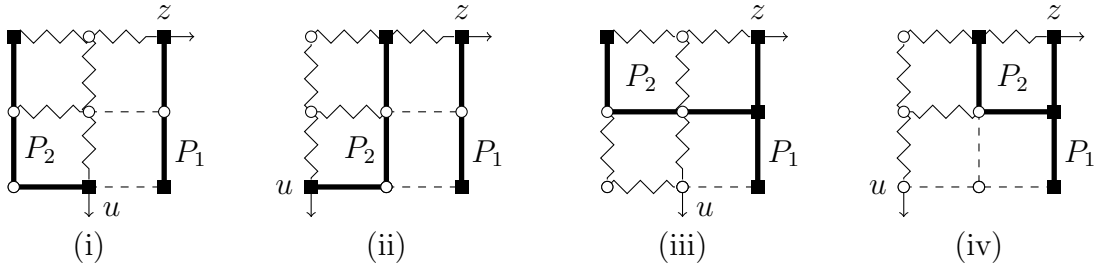 

The two leftmost pictures in Fig.\ref{inB} cover the cases $t_2\in A\setminus B$. For $s_2\in B(1)$, $P_2$ is the  path in $B(1)\cup (A\setminus B)$, see (i); for 
 $s_2\in B(2)$,  $P_2$ is the path in $B(2)\cup (A\setminus B)$, see (ii).
The two rightmost pictures in Fig.\ref{inB} cover the cases $t_2\in B\setminus A$. For $s_2\in B(1)$,  $P_2$ is the unique path in $B(1)\cup A(2) \cup B$, $P_1$, see (iii); for 
 $s_2\in B(2)$, $P_2$ is contained in $B(2)\cup A(2)\cup B$, see (iv).
\\

To finish the proof  we 
discuss $\|S\|=2$, $3$ and $4$ separately, and we assume that none of the cases b, c, and d
 applies.\\

Let $\|S\|=2$. First assume that  $S$ has two singletons, $s_3,s_4\in S$. Then we have $\pi_1,\pi_2\subset L$, 
and by Cases c and d,
we may assume that $s_1=(2,3)$, $t_1=(3,1)$ or $(3,2)$, and the only free vertex of $L$ is $w=(3,3)$. The pictures in Fig.\ref{pinAB} show the linkage  $P_1$ for $\pi_1$ and 
the AA-clip $CL(t_1,w)$ to escape $s_3,s_4$.

\begin{figure}[htp]
  \tikzstyle{T} = [rectangle, minimum width=.1pt, fill, inner sep=2.5pt]
\tikzstyle{B} = [circle, draw=black!, minimum width=1pt, fill=white, inner sep=1.5pt]
\tikzstyle{txt}  = [circle, minimum width=1pt, draw=white, inner sep=0pt]
\tikzstyle{Wedge} = [draw,line width=1.5pt,-,black!100]
\tikzstyle{M} = [circle, draw=black!, minimum width=1pt, fill=white, inner sep=1pt]
\begin{center}
	\begin{tikzpicture}
           
\draw[dashed]  (2,1) --(3,1) (1,2)--(2,2); 
 \draw[snake]  (1,3)--(1,1.1) (1,1)--(1.9,1) (1,3)--(2.9,3) (3,3)--(3,1.1) (2,2)--(2,2.9);
 \draw[->]  (2,1) -- (2,.6); \draw[->]  (3,1) -- (3,.6);
  
  \draw[line width=2pt] (2,1)--(2,2)--(3,2) ;  
         
\foreach \y in{3} \node[T]() at (3,\y){};  
   \foreach \x in{1,2}  \foreach \y in{1,2,3} \node[B]() at (\x,\y){};      
                         
 \node[T,label=right:$s_1$]() at (3,2){};  
  \node[T](t1) at (2,1){};         \node() at (1.7,.7){$t_1$};  

\node[T]() at (3,3){};    \node[T]() at (1,1){};
\node[B,label=right:$w$]() at (3,1){}; 
                                      
            \node() at (2.35,1.65){$P_1$};                       
     	\end{tikzpicture}		
		\hskip1truecm
	\begin{tikzpicture}
           
\draw[dashed]  (2,1) --(3,1) (1,2)--(2,2); 
 \draw[snake]  (1,3)--(1,1.1)  (1,3)--(2.9,3) (3,3)--(3,1.1) (2,2)--(2,2.9);
 \draw[->]  (1,1) -- (1,.6); \draw[->]  (3,1) -- (3,.6);
  
  \draw[line width=2pt] (1,1)--(2,1)--(2,2)--(3,2) ;  
         
\foreach \y in{3} \node[T]() at (3,\y){};  
   \foreach \x in{1,2}  \foreach \y in{1,2,3} \node[B]() at (\x,\y){};      
                         
 \node[T,label=right:$s_1$]() at (3,2){};  
  \node[T]() at (2,1){};      
\node[T]() at (3,3){};    \node[T,label=left:$t_1$]() at (1,1){};    \node[B,label=right:$w$]() at (3,1){};

            \node() at (1.65,1.5){$P_1$};                       
     	\end{tikzpicture}\end{center}
	\caption{$\pi_1,\pi_2\subset A\cup B$ }
	\label{pinAB}
\end{figure}
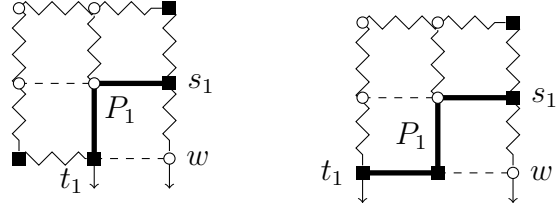

Next we assume that $S$ contains no singleton,
$s_1,s_2\in S$. If  $(3,3)\in\{t_1,t_2\}$, then we shift  $(3,3)\mapsto (3,2)$ even if $(3,2)$
is not a free vertex. Using that 
$Q^* =Q-(3,3)$ is  weakly $2$-linked, by Lemma \ref{w2linked}, there is a linkage in $Q^*$ for  $\pi_1$ and $\pi_2$. If  $B\setminus A=\{s_3,s_4\}$, then 
we need another shift $(2,3)\mapsto (3,3)$ 
to make $(2,3)$ a free vertex (see an example in Fig.\ref{last2} (i)). 
 \begin{figure}[htp]
  \tikzstyle{T} = [rectangle, minimum width=.1pt, fill, inner sep=2.5pt]
\tikzstyle{B} = [circle, draw=black!, minimum width=1pt, fill=white, inner sep=1.5pt]
\tikzstyle{txt}  = [circle, minimum width=1pt, draw=white, inner sep=0pt]
\tikzstyle{Wedge} = [draw,line width=1.5pt,-,black!100]
\tikzstyle{M} = [circle, draw=black!, minimum width=1pt, fill=white, inner sep=1pt]
\begin{center}
		
	\begin{tikzpicture}
 
  \foreach \x in{1,2}  \draw[line width=2pt] (\x,1)--(\x,3)  ;
  \foreach \y in{2,3} \draw[line width=2pt] (1,\y)--(3,\y)  ; 
\draw[line width=2pt]  (1,1)--(2,1) (3,2)--(3,3); 
   
  \draw[->,line width=1pt] (3,2)--(3,1.1);
    \draw[->,line width=1pt] (3,1)--(2.1,1);   
    
\node[B]() at (1,2){};  \node[B]() at (2,2){};   \node[B]() at (2,3){};     
 \node[B]() at (1,3){};  \node[B]() at (1,1){}; 

  \node[T,label=right:$s_4$](s4) at (3,3){}; 
  \node[T,label=above:$s_2$] (s2) at (2,3) {};             

         \node[T,label=below:$t_2$](t2) at (3,1){};       
            \node[T,label=below:$t_1$] (t1) at (2,1) {};   
            \node[T,label=left:$s_1$] (s1) at (1,2) {};          
                 \node[T,label=right:$s_3$] (s3) at (3,2) {};   
               
    \node[txt]() at (1.5,2.5){$Q^*$};   
      \node[txt]() at (2,0){(i)};     
     	\end{tikzpicture}
			\hskip.5truecm
	\begin{tikzpicture}
           
\draw[dashed]  (1,2) --(3,2) (2,3)--(3,3); 
 \draw[snake]  (1,3)--(1,1.1) (1,3)--(1.9,3) (2,3)--(2,1.1)  ;
 \draw[->]  (2,1) -- (2,.6); \draw[->]  (1,1) -- (1,.6);
  
   \draw[line width=2pt] (1,1)--(3,1)--(3,3);  
         
\foreach \y in{1,3} \node[T]() at (3,\y){};  
   \foreach \x in{1,2}  \foreach \y in{1,2,3} \node[B]() at (\x,\y){};      
                         
 \node[T]() at (3,2){};  
  \node[B](w) at (2,1){};          
  \node[T]() at (1,1){};   \node[B]() at (2,1){};  
 \node[T,label=above:$t_2$]() at (3,3){};  
         
           \node[T,label=left:$s_2$](s2) at (1,1){};                        
            \node() at (2.25,.75){$w$};   
           \node() at (2.65,1.35){$P_2$};         
             \node() at (2,0){(ii)};            
     	\end{tikzpicture}
			\hskip.5truecm	
	\begin{tikzpicture}
           
\draw[dashed]  (1,3)--(2,3) --(3,3) --(3,2)--(2,2) (2,1)--(3,1); 
 \draw[snake]  (2,2)--(1,2) (1,2)--(1,2.9)  (1,2)--(1,1.1) ;
  \draw[->]  (1,1) -- (1,.6);
  \draw[->,line width=1pt]  (3,2) -- (3,1.1);
   \draw[line width=2pt] (1,1)--(2,1)--(2,3);  
         
   \foreach \x in{1,2}  \foreach \y in{1,2} \node[B]() at (\x,\y){};      
                         
 \node[T]() at (3,2){};   
  \node[T,label=below:$s_2$]() at (2,1){};          
  \node[T]() at (1,1){};      \node[T,label=above:$s_1$]() at (2,3){};      \node[B]() at (1,3){};    
 \node[T,label=above:$t_2$]() at (3,3){};  
  \node[B,label=right:$w$]() at (3,1){};  
       
           \node[T,label=left:$t_1$](t1) at (1,1){};                   
            \node() at (1.65,1.35){$P_1$};        
             \node() at (2,0){(iii)};            
     	\end{tikzpicture}
			\hskip.5truecm
	\begin{tikzpicture}
           
\draw[dashed]  (1,3)--(1,2) --(1,1) --(2,1) (2,2)--(3,2); 
 \draw[snake]  (1,2)--(1.9,2) (2,3)--(2,1.1) (2,1)--(2.9,1.1) ;
  \draw[->]  (3,1) -- (3,.6);
  
   \draw[line width=2pt] (1,3)--(3,3)--(3,1);  
         
   \foreach \x in{1,2}  \foreach \y in{1,2} \node[B]() at (\x,\y){};      
                         
 \node[B,label=right:$w$]() at (3,2){};   
  \node[T,label=below:$s_2$]() at (2,1){};       \node[T,label=right:$t_1$]() at (3,1){};      
  \node[T,label=below:$s_4$]() at (1,1){};      \node[T,label=above:$s_1$]() at (1,3){};      
 \node[T,label=above:$t_2$]() at (3,3){};  
         \node[B]() at (2,3){};  
            
            \node() at (2.65,2.65){$P_1$};        
             \node() at (2,0){(iv)};            
     	\end{tikzpicture}
	\end{center}
	\caption{Solutions for $\|S\|=2$}
	\label{last2}
\end{figure}
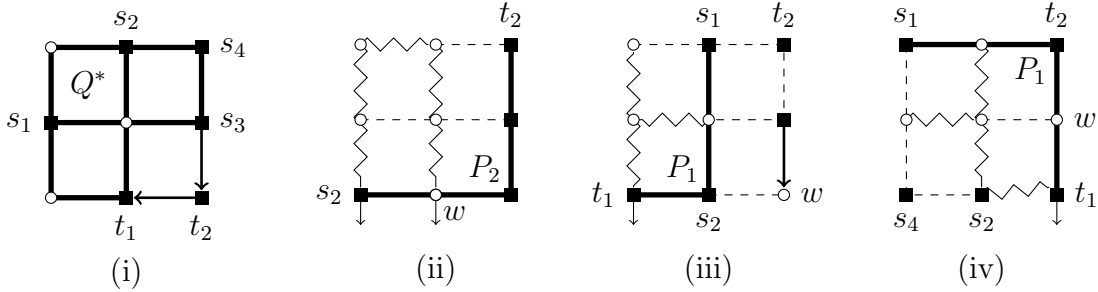

Assume now that $S$ contains one singleton, let $s_1,s_3\in S$ and let $w$ be the free vertex in $L$. 
By Cases c and d, we may assume that $s_2\in A\setminus B$ and $t_2\in B\setminus A$. 
If $w\in A\setminus B$, then there is a linkage in $L$ for $\pi_2$ and an AA-clip $CL(s_2,w)$ yields a solution as shown 
in Fig.\ref{last2}(ii). 
If $t_1\in  A\setminus B$, then we take a linkage $P_1$ for $\pi_1$ in the union of $A$ and the column containing $s_1$. Then there is an edge disjoint mating path
in $Q-B$
from any location of $s_3$  into $t_1$. Furthermore, we need to apply the shift  
 $(2,3)\mapsto w$, provided $\|B\setminus A\|=2$
 (for an example see Fig.\ref{last2} (iii)). 
 
Thus we conclude that $w,t_1,t_2\in B$ and $A\setminus B=\{s_2,s_4\}$.
 Then we take a linkage $P_1$ for $\pi_1$ in the union of $B$ and the row 
 that contains $s_1$, and there is an edge disjoint mating path from $s_3$ to $(3,3)$ (see Fig.\ref{last2} (iv)).

For $\|S\|=3$, w.l.o.g. we may assume that either $s_1,s_2,s_3\in S$ or $s_2,s_3,s_4\in S$.
In each case $L$ has two free vertices. First let $s_1,s_2,s_3\in S$.

(i) If $\|B\setminus A\|=0$, then we take a linkage $P_1$
 for $\pi_1$ in the union of $A$ and the column containing $s_1$. In the complement  of $P_1$
there is an AB-clip $CL(t_1,z)$, where $z=(1,3)$ (see an example
in Fig.\ref{last3}(i)). 

(ii)  If $\|B\setminus A\|=1$ and 
$B\setminus A$ has a terminal among $t_1,t_2$, say $t_1$, then 
we take a linkage  $P_1$ for $\pi_1$ 
in the union of $B$ and the row 
 that contains $s_1$. Since there is a free vertex $u\in A$, an 
 AB-clip $CL(u,t_1)$ remains in the complement of $P_1$
 (see an example  in Fig.\ref{last3}(ii)).

 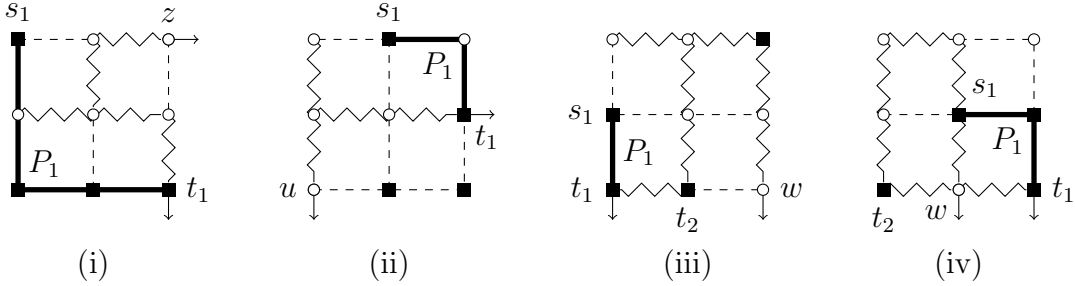
\begin{figure}[htp]
  \tikzstyle{T} = [rectangle, minimum width=.1pt, fill, inner sep=2.5pt]
\tikzstyle{B} = [circle, draw=black!, minimum width=1pt, fill=white, inner sep=1.5pt]
\tikzstyle{txt}  = [circle, minimum width=1pt, draw=white, inner sep=0pt]
\tikzstyle{Wedge} = [draw,line width=1.5pt,-,black!100]
\tikzstyle{M} = [circle, draw=black!, minimum width=1pt, fill=white, inner sep=1pt]
\begin{center}
\begin{tikzpicture}            
\draw[dashed]   (1,3)--(2,3) (3,3)--(3,2) (2,1)--(2,2); 

\draw[snake] (2,3)--(2.9,3);\draw[->] (3,3)--(3.4,3);
\draw[snake] (2,3)--(2,2.1) (1,2)--(2.9,2) (3,2)--(3,1.1);\draw[->] (3,1)--(3,.6);
\draw[line width=2pt] (1,3)--(1,1)--(3,1) ;    
      
\foreach \x in {1,2,3}\foreach \y in{1,2,3} \node[B]() at (\x,\y){};                                

            \node[T,label=right:$t_1$] (t1) at (3,1) {};     
            \node[T] () at (2,1) {};    \node[T] () at (1,1) {};  
            \node[T,label=above:$s_1$] (s1) at (1,3) {};   
             \node[B,label=above:$z$] (z) at (3,3) {};         
                    
\node() at (1.35,1.35){$P_1$};     
      \node() at (2,0){(i)};     
     	\end{tikzpicture}
	\hskip.5truecm	
	\begin{tikzpicture}            
\draw[dashed]   (2,3)--(1,3) (2,1)--(2,3) (1,1)--(3,1)--(3,2); 

\draw[snake] (1,2)--(2.9,2); \draw[->] (3,2)--(3.4,2);
\draw[snake] (1,3)--(1,1.1) ;    \draw[->] (1,1)--(1,.6);
    \draw[line width=2pt]   (2,3)--(3,3)--(3,2);
    
\foreach \x in {1,2,3}\foreach \y in{1,2,3} \node[B]() at (\x,\y){};                                

            \node[T] (t1) at (3,2) {};    \node () at (3.3,1.7) {$t_1$};
            \node[T] () at (2,1) {};    \node[T] () at (3,1) {};  
            \node[T,label=above:$s_1$] (s1) at (2,3) {};   
             \node[B,label=left:$u$] (u) at (1,1) {};    
                    
\node() at (2.65,2.65){$P_1$};     
      \node() at (2,0){(ii)};     
     	\end{tikzpicture}
	\hskip.5truecm	
		\begin{tikzpicture}
           
\draw[dashed]  (2,1) --(3,1) (1,3)--(1,2)--(3,2); 
 \draw[snake]  (2,3)--(2,1.1) (2,1)--(1.1,1) (1,3)--(2,3)--(2.9,3) (3,3)--(3,1.1);
 \draw[->]  (1,1) -- (1,.6); \draw[->]  (3,1) -- (3,.6);
  
   \draw[line width=2pt] (1,1) --(1,2);  
         
   \foreach \x in{1,2}  \foreach \y in{1,2,3} \node[B]() at (\x,\y){};      
                         
  \node[B,label=right:$w$](w) at (3,1){};     
\node[T,label=left:$s_1$]() at (1,2){};    \node[T,,label=left:$t_1$]() at (1,1){};  
        \node[T,label=below:$t_2$]() at (2,1){};      
        \node[T]() at (3,3){};   
          \node[B]() at (3,2){};                                
                \node() at (2,0){(iii)};       
            \node() at (1.35,1.5){$P_1$};                       
     	\end{tikzpicture}
	\hskip.5truecm				
	\begin{tikzpicture}
           
\draw[dashed]  (2,3) --(3,3) --(3,2) (1,2)--(2,2); 
 \draw[snake]  (1,3)--(1,1.1) (1,1)--(1.9,1) (1,3)--(2,3)--(2,1.1) (2,1)--(2.9,1);
 \draw[->]  (2,1) -- (2,.6); \draw[->]  (3,1) -- (3,.6);
  
   \draw[line width=2pt] (2,2) --(3,2)--(3,1) ;  
         
\foreach \y in{1,2} \node[T]() at (3,\y){};  
   \foreach \x in{1,2}  \foreach \y in{1,2,3} \node[B]() at (\x,\y){};      
                         
 \node[T,label=right:$t_1$]() at (3,1){};   
  \node[B](w) at (2,1){};         \node() at (1.7,.7){$w$};  
\node[B]() at (3,3){};    \node[T,label=below:$t_2$]() at (1,1){};  

\node[T]() at (2,2){};  \node() at (2.35,2.35){$s_1$};   
                                   
                \node() at (2,0){(iv)};       
            \node() at (2.65,1.65){$P_1$};                       
     	\end{tikzpicture}
			
\end{center}
	\caption{Solutions for  $\|S\|=3$ }
	\label{last3}
\end{figure}

(iii) If $\|B\setminus A\|=1$, and (ii) is not true, then $t_1,t_2\in A$.
If $w=(3,3)$ is a free vertex, then take a linkage  $P_1$
 for $\pi_1$ in the union of $A$ and the column containing $s_1$.
An AA-clip $CL(w,t_1)$ remains in the complement  of $P_1$ as is shown 
in  Fig.\ref{last3}(iii). 
If $(3,3)$ is not free, say $t_1=(3,3)$,  then we
take a linkage $P_1$ for $\pi_1$ in the union of $B$ and the row containing $s_1$.  In the complement of $P_1$ there is an AA-clip
$CL(w,t_1)$, where $w\in A\setminus B$ is a free vertex  (see 
 Fig.\ref{last3} (iv)).

 If  $\|B\setminus A\|=2$, then  $B\setminus A$ contains  a terminal among $t_1,t_2$, say $t_1$, and $A$ contains two free vertices.
Now we take a linkage $P_1$ for $\pi_1$ 
in the union of $B$ and the row 
 that contains $s_1$. Since there are two  free vertices in $A$, an 
 AA-clip remains in the complement of $P_1$. 

For  $s_1,s_3,s_4\in S$  we may assume that
$s_2\in A\setminus B$ and $t_2\in B\setminus A$, by Cases c and d. 
This implies that $\|B\setminus A\|\geq 1$. If $t_1\in B\setminus A$, then we have the solution as
above; if $t_1\in A$, then either $t_1=(3,3)$ or 
$w=(3,3)$ is a free vertex, 
 thus we obtain a solution as in (iii) above.\\
 
 Let $\|S\|=4$. By Case b, there is no pair in $S$, let $s_i\in S$, for $1\leq i\leq 4$, and
 $t_1,t_2\in L$. Since there are four edges between $S$ and $L$ each must be used by some linkage or mating path. Consider the `diagonal partition' of the vertices of $S$
 into $S_1=\{(1,1),(2,2)\}$ and $S_2=\{(2,1),(1,2)\}$; both vertices in $S_i$, $i=1,2$, exit from $S$ to the same part of $L$, either to $A\setminus B$ or to $B\setminus A$
 according to the following rule. 
 
 If $t_1,t_2\in A\setminus B$, then the terminals in $S_i$ containing $s_1$ exit to $A$, the vertices of the other diagonal exit  to $B$. Let $P$ be the (unique) path in $B(1)\cup A\cup B(2)$ between $s_1$ and its diagonal pair $s_\ell\in S_i$ ($2\leq \ell\leq 4$).  Now $P$ is the edge disjoint union of an $s_1,t_1$-path $P_1$ and a path mating $s_\ell$ to $t_1$ where it escapes from  $Q$. The vertex in $A(1)$ of the other diagonal is mated the  to the free vertex $(1,3)$ and the the other terminal of the diagonal is mated along $A(2)\cup B$ to the free vertex $(3,3)$, where they escape from $Q$ (see Fig.\ref{S4} (i)).

 If $t_1,t_2\in B\setminus A$, then the terminals of the diagonal $S_i$ containing the pair of the terminal at $(2,3)$ exit $S$ to $B$, and the vertices of the other diagonal exit $S$ through 
 the free vertices in $A\setminus B$.  For $t_2=(2,3)$, let $P$ be the (unique) path in $A(1)\cup B\cup A(2)$ between $s_2$ and its diagonal pair $s_\ell$.  Now $P$ is the edge disjoint union of an $s_2,t_2$-path $P_2$ and a path $P_0$ mating $s_\ell$ to $t_2$.  We extend $P_0$ to the free vertex $(3,3)$, where it exits $Q$ (see Fig.\ref{S4} (ii)).
  
 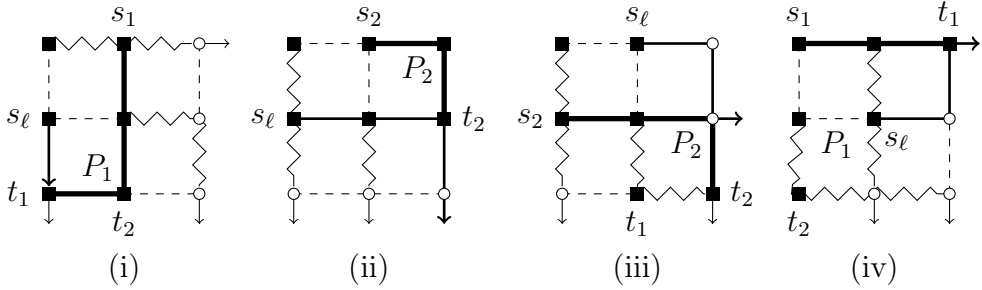
\begin{figure}[htp]
 \tikzstyle{T} = [rectangle, minimum width=.1pt, fill, inner sep=2.5pt]
\tikzstyle{B} = [circle, draw=black!, minimum width=1pt, fill=white, inner sep=1.5pt]
\tikzstyle{M} = [circle, draw=black!, minimum width=1pt, fill=white, inner sep=1pt]
\begin{center}
\begin{tikzpicture}            
\draw[dashed]   (1,3) -- (1,2)--(2,2) (3,3)--(3,2) (2,1)--(3,1); 
\draw[->,line width=1pt] (1,2)--(1,1.1);
\draw[->] (1,1)--(1,.6);
\draw[snake] (1,3)--(2.9,3);\draw[->] (3,3)--(3.4,3);
\draw[snake] (2,2)--(3,2)--(3,1.1);\draw[->] (3,1)--(3,.6);
\draw[line width=2pt] (1,1)--(2,1)--(2,3) ;    
      
\foreach \x in {1,2}\foreach \y in{1,2,3} \node[T]() at (\x,\y){};                                
\foreach \y in{1,2,3} \node[B]() at (3,\y){};

            \node[T,label=left:$t_1$] (t1) at (1,1) {};    
            \node[T,label=left:$s_\ell$] (sell) at (1,2) {};    
            \node[T,label=above:$s_1$] (s1) at (2,3) {};          
               \node[T,label=below:$t_2$] (t2) at (2,1) {};      
\node() at (1.65,1.35){$P_1$};     
      \node() at (2,0){(i)};     
     	\end{tikzpicture}
	\begin{tikzpicture}            
\draw[dashed]   (1,1) -- (3,1)  (1,3)--(2,3) --(2,2); 
\draw[snake] (1,3)--(1,1.1);\draw[->] (1,1)--(1,.6);
\draw[->,line width=1pt] (1,2)--(3,2)--(3,.6);
\draw[snake] (2,2)--(2,1.1);\draw[->] (2,1)--(2,.6);
\draw[line width=2pt] (2,3)--(3,3)--(3,2) ;    
      
\foreach \x in {1,2,3}\foreach \y in{2,3} \node[T]() at (\x,\y){};                                
\foreach \x in{1,2,3} \node[B]() at (\x,1){};

            \node[T,label=right:$t_2$] (t2) at (3,2) {};    
            \node[T,label=left:$s_\ell$] (t1) at (1,2) {};    
            \node[T,label=above:$s_2$] (s2) at (2,3) {};          
                    
\node() at (2.65,2.65){$P_2$};     
      \node() at (2,0){(ii)};     
     	\end{tikzpicture}
\begin{tikzpicture}            
\draw[dashed]   (1,1) -- (2,1) (1,3)--(2,3)--(2,2); 
\draw[snake] (1,3)--(1,1.1);\draw[->] (1,1)--(1,.6);
\draw[->,line width=1pt] (2,3)--(3,3)--(3,2)--(3.4,2);
\draw[snake] (2,2)--(2,1)--(2.9,1);\draw[->] (3,1)--(3,.6);
\draw[line width=2pt] (1,2)--(3,2)--(3,1) ;    
      
 \node[T]() at (1,3){};     \node[T]() at (2,2){};     
 \node[T,label=below:$t_1$]() at (2,1){};                                 
 \node[B]() at (1,1){}; \node[B]() at (3,3){}; \node[B]() at (3,2){};

            \node[T,label=right:$t_2$] (t2) at (3,1) {};    
            \node[T,label=above:$s_\ell$] (sl) at (2,3) {};    
            \node[T,label=left:$s_2$] (s2) at (1,2) {};          
                    
\node() at (2.65,1.65){$P_2$};     
      \node() at (2,0){(iii)};     
     	\end{tikzpicture}		
\begin{tikzpicture}            
\draw[dashed]   (3,1) -- (3,2)  (1,3)--(1,2)--(2,2); 
\draw[snake] (2,3)--(2,1.1);\draw[->] (2,1)--(2,.6);
\draw[->,line width=1pt] (2,2)--(3,2)--(3,3)--(3.4,3);
\draw[snake] (1,2)--(.9,1)--(3,1);\draw[->] (3,1)--(3,.6);
\draw[line width=2pt] (1,3)--(3,3);    
      
 \node[T,label=below:$t_2$]() at (1,1){};     \node[T]() at (1,2){};     \node[T]() at (2,3){};                          
 \node[B]() at (2,1){}; \node[B]() at (3,1){}; \node[B]() at (3,2){};

            \node[T,label=above:$t_1$] (t1) at (3,3) {};    
            \node[T](sl) at (2,2) {};    \node()at(2.3,1.7){$s_\ell$};
            \node[T,label=above:$s_1$] (s1) at (1,3) {};          
                    
\node() at (1.5,1.65){$P_1$};     
      \node() at (2,0){(iv)};     
     	\end{tikzpicture}		
\end{center}
	\caption{Solutions for  $\|S\|=4$ }
	\label{S4}
\end{figure}

Assume now that $\|B\setminus A\|=0$, and $(3,3)$ is a terminal, say $t_2=(3,3)$. We begin as before, the vertices of the diagonal $s_2, s_\ell\in S_i$  exit $S$ to $B$ and we form an $s_2,s_\ell$-path
$P$ in $A(1)\cup B\cup A(2)$. Now $s_\ell$ is mated along $P$ to the free vertex $(2,3)$ where it escapes from $Q$. The other part of $P$ from $s_2$ is extended to $t_2$ thus completing a linkage $P_2$ and making $(3,3)$ a free vertex. Then the vertices of the other diagonal can be mated in
$B(1)\cup A\cup B(2)$ into $(3,3)$ and into the free vertex of  $A\setminus B$ 
(see Fig.\ref{S4} (iii)).

Finally assume that $B\setminus A$ contains one terminal, say $t_1\in B\setminus A$.
As before, the vertices of the diagonal $s_1, s_\ell$  exit $S$ to $B$ and form an $s_1,s_\ell$-path $P$ in $A(1)\cup B\cup A(2)$. Now $P$ is the edge disjoint union of an $s_1,t_1$-path $P_1$ and a path $P_0$ mating $s_\ell$ to $t_1$, where it escapes from $Q$.  
Since there are two free vertices on $A$, the terminals of the other diagonal can be mated into them in
$B(1)\cup A\cup B(2)$  where they escape from $Q$ 
(see Fig.\ref{S4} (iv)).
    \end{proof}    
       \begin{lemma}
   \label{heavy5}
For any set of  five terminals, $\{s_1,t_1,s_2,s_3,s_4\}\subset Q$,
there is
 an $s_1,t_1$-path $P_1\subset Q$, and the complement of $P_1$ contains edge disjoint escape paths from $s_2,s_3,s_4$ into three distinct exit vertices of $L$ such that $B\setminus A$ contains at most one exit.
   \end{lemma}
    \begin{proof} 
We use clips and shifting 
along $L$ as defined in Section \ref{tools}. Let $S=Q-L$.

Case a: $\pi_1\subset S$. Assuming that $\|S\|=2$, all singletons are in $L$. If $\|B\setminus A\|=2$, then we form a pair $\pi_0=\{z,w\}$,where $z\in B\setminus A$ is a terminal and $w\in A$ is a free vertex.
Since $Q$ is $2$-path-pairable, there is a linkage $P_1$ and $P_0$, for $\pi_1$ and $\pi_0$, respectively. Then $P_0$ is an escape path of $z$ to $w\in A$ making $z$ a free vertex in $B\setminus A$.

For $\|S\|=3$, let $s_1,t_1,s_2\in S$. If $\|B\setminus A\|=2$, then we make $(2,3)\in B\setminus A$ a free vertex by shifting $(2,3)\mapsto (3,3)$.   
Then we form a pair $\pi_0=\{s_2,w\}$, where $w\in A\setminus B$ is a free vertex.
Since    $Q^\prime =Q-(3,3)$ is $2$-path pairable, there is  a linkage $P_1$, $P_0$ in $Q^\prime$ for  $\pi_1$, $\pi_0$. Then $P_0$ is an escape path for $s_2\in S$ to $w\in A$, the remaining singletons are in $L$.

For $\|S\|=4$, two terminals must escape from $S$ and there is one terminal in $L$. Therefore, either both vertices $u=(3,1)$ and $w=(3,2)$ are free for escaping, or $\|B\setminus A\| =\emptyset$ and thus
$z=(1,3)$ is available for escaping with a free
vertex among $u,w$.

 Let $s_\ell=(1,1)$, $1\leq \ell \leq 4$.
If $\ell\neq 1$, then take a linkage in $S-s_\ell$ for $\pi_1$. Escape paths for $s_\ell$ and the fourth terminal in $S$ 
to any two of $u,w,z$ are obtained easily along the $8$-cycle on $Q-(3,3)$ (see
Fig.\ref{5S4}(i)). If $\ell=1$, then take a linkage in $S$ for $\pi_1$, then escape paths  to any two of $u,w,z$ exist in the complement of $S$ (see Fig.\ref{5S4}(ii)).

	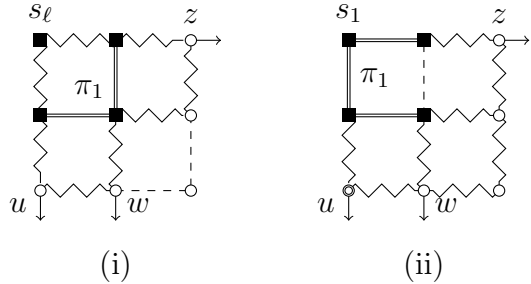
\begin{figure}[htp]
 \tikzstyle{T} = [rectangle, minimum width=.1pt, fill, inner sep=2.5pt]
\tikzstyle{B} = [circle, draw=black!, minimum width=1pt, fill=white, inner sep=1.5pt]
\tikzstyle{M} = [circle, draw=black!, minimum width=1pt, fill=white, inner sep=1pt]
\begin{center}
\begin{tikzpicture}
           
\draw[dashed]  (2,1) --(3,1)--(3,2); 
 \draw[snake]  (1,3) -- (1,1.1) (1,1)--(2,1)--(2,2)--(2.9,2)
(1,3)--(2.9,3) (3,2)--(3,3);
  \draw[->]  (1,1) -- (1,.6); \draw[->]  (3,3) -- (3.4,3);
   \draw[->]  (2,1) -- (2,.6);
   \draw[double] (2,3) -- (2,2) -- (1,2);  
        
\foreach \x in {1,2}\foreach \y in{2,3} \node[T]() at (\x,\y){};                                
  \node[B]() at (1,1){}; \node[B]() at (3,1){};  \node[B]() at (2,1){};
           \node[B]() at (3,2){};      \node() at (2.3,.8){$w$}; 
           
\node[T,label=above:$s_\ell$]() at (1,3){};            
  \node[B,label=above:$z$](z) at (3,3){};   
     
                \node() at (.7,.8){$u$};                           
            \node() at (1.65,2.35){$\pi_1$};     
      \node() at (2,0){(i)};     
                                
     	\end{tikzpicture}
	\hskip1cm	
	\begin{tikzpicture}
           
\draw[dashed]   (2,2) --(2,3); 
 \draw[snake]  (2,2)--(3,2)--(3,1)--(2,1) (1,1)--(2,1) --(2,2)
 (1,2)--(1,1.1) (2,3)--(3,3)--(3,2);
  \draw[->]  (1,1) -- (1,.6); \draw[->]   (3,3)  --(3.4,3);
  \draw[->]  (2,1) -- (2,.6);
   \draw[double] (2,3) -- (1,3) -- (1,2)--(2,2);  
        
\foreach \x in {1,2}\foreach \y in{2,3} \node[T]() at (\x,\y){};                                

\node[B]() at (2,1){};  \node[B]() at (1,1){};  
\node[B]() at (3,2){};  \node[B]() at (3,1){};  
   \node[B,label=above:$z$](z) at (3,3){};            
    \node[T,label=above:$s_1$](s1) at (1,3){};  
     \node() at (2.3,.8){$w$}; 
   \node[M]() at (1,1){};           
         \node() at (.7,.8){$u$};                           
            \node() at (1.35,2.5){$\pi_1$};     
      \node() at (2,0){(ii)};     
                                
     	\end{tikzpicture}

	\end{center}
	\caption{$\pi_1\subset S$ and $\|S\|=4$}
	\label{5S4}

	\end{figure}

Case b:  $\|S\|\leq 1$. For $\|S\|=0$, we define a linkage in $L$ for $\pi_1$. If $\|B\setminus A\|=2$ and $t_1\notin B\setminus A$,  then we make  $(2,3)$ free by shifting along the path through $c=(3,3)$ until the first free vertex in $A$ (that might be $t_1$). For $\|S\|=1$ and $\pi_1\subset L$ we do the same as before, then the terminal in $S$ escapes to $s_1$. 

Let $s_1\in S$. If  $t_1\in B\setminus A$, then there is a linkage for $\pi_1$
in the union of $B$ and the row containing $t_1$. If $t_1\in A$ and
$\|B\setminus A\|=2$, then first we make $(2,3)$ free by shifting terminals along the path through $c$ until the first free vertex in $A$, then the terminal in $S$ escapes to $t_1$ (or its shift).
 
\begin{figure}[htp]
  \tikzstyle{T} = [rectangle, minimum width=.1pt, fill, inner sep=2.5pt]
\tikzstyle{B} = [circle, draw=black!, minimum width=1pt, fill=white, inner sep=1.5pt]
\tikzstyle{txt}  = [circle, minimum width=1pt, draw=white, inner sep=0pt]
\tikzstyle{Wedge} = [draw,line width=1.5pt,-,black!100]
\tikzstyle{M} = [circle, draw=black!, minimum width=1pt, fill=white, inner sep=1pt]
\begin{center}
\begin{tikzpicture}    
 \draw[dashed]   (2,2)--(2,1) (1,1)--(1,0)--(2,0); 
 \draw[snake]  (1,0)--(0,0) -- (0,2)-- (1.9,2) (0,1)-- (2,1)--(2,0.15)  (1,1) -- (1,2) ; 
 \draw[->] (1,0)--(1,-0.4); \draw[->]  (2,0)--(2,-0.4);
 \draw[->]  (2,2)--(2.4,2);
  \draw[line width=2pt] (1,0) -- (2,0);  
 \foreach \x in{0,1,2}\foreach \y in{0,1,2}\node[B]() at (\x,\y){};                        

\node[B] (z) at (2,2) {};        \node()at(2,2.25){$z$};     
\node[T] (s1) at (1,0) {};    \node()at(0.7,-0.3){$s_1$};
\node[T] (t1) at (2,0) {};     \node()at(2.3,-0.3){$t_1$};
         
           \node[txt]() at (1.5,.35){$P_1$}; 
            \node() at (.5,1.5){$H$};   \node() at (1,-.85){(i)}; 
     	\end{tikzpicture}
	\hskip.5cm
\begin{tikzpicture}
 \draw[dashed]   (1,2)--(2,2)--(2,1)-- (0,1); 
  \draw[snake]  (0,0) -- (0,2)--(1,2)  (1,1)--(1,0)
  (1,1) -- (1,2)  ; 
  \draw[line width=2pt] (2,0) -- (2,2);  
   \draw[->,line width=1pt] (0,0)--(.9,0) (1,0)--(1.9,0);
   \draw[->](0,0)--(0,-0.4); \draw[->](1,0)--(1,-0.4);
   \foreach \x in{0,1,2}\foreach \y in{0,1,2}\node[B]() at (\x,\y){};                        
  
\node[T]() at (1,2){}; 
 \node() at (-0.3,-0.3){$u$}; 
 \node() at (1.3,-0.3){$w$}; 
\node[T]() at (0,1){};      
           \node() at (2,-0.3){$c$}; 
          \node[T,label=right:$t_1$] (t1) at (2,2) {};                  
          \node[T,label=right:$s_1$] (s1) at (2,0) {};   
           \node[txt]() at (1.65,1.5){$P_1$}; 
           \node() at (1,-.85){(ii)}; 
     	\end{tikzpicture}
	\hskip.5cm
	\begin{tikzpicture}
   
 \draw[dashed]   (2,1)-- (1,1) (0,0)--(2,0)--(2,1); 
  \draw[snake]  (0,0) -- (0,2)--(1,2)  (0,1)-- (1,1)--(1,0)
  (1,1) -- (1,2) --(2,2) ; 
  \draw[line width=2pt] (2,2) -- (2,1);  
  \draw[->](0,0)--(0,-0.4); \draw[->](1,0)--(1,-0.4);
   \draw[->]  (2,2)--(2.4,2);
    \foreach \x in{0,1,2}\foreach \y in{0,1,2}\node[B]() at (\x,\y){};                        
      
\node[T]() at (1,2){}; 
  \node() at (-0.3,-0.3){$u$}; 
  \node() at (1.3,-0.3){$w$}; 
\node[T]() at (0,1){};      
\node[T] () at (0,2) {};  
    
          \node[T,label=right:$t_1$] (t1) at (2,1) {};          
            \node()at(2.3,1.7){$s_1$};           
          \node[T,label=above:$z$] () at (2,2) {}; 
          \node[txt]() at (1.65,1.5){$P_1$};  \node() at (.5,1.5){$H$}; 
            \node() at (1,-.85){(iii)}; 
                  	\end{tikzpicture}
	\hskip.5cm
	\begin{tikzpicture}
   
 \draw[dashed]   (2,1)-- (1,1)--(1,2) (0,1)--(0,0)--(2,0)--(2,1); 
  \draw[snake]  (0,1) -- (0,2)--(1,2)  (0,1)-- (1,1)--(1,0)
   (1,2) --(2,2) ; 
  \draw[line width=2pt] (2,2) -- (2,0)--(1,0);  
  \draw[->](1,0)--(1,-0.4);
   \draw[->]  (2,2)--(2.4,2);
    \foreach \x in{0,1,2}\foreach \y in{0,1,2}\node[B]() at (\x,\y){};                        
      
\node[T]() at (1,2){}; 
  \node() at (1.3,-0.3){$t_1$}; 
\node[T]() at (0,1){};      
\node[T] () at (0,0) {};  \node[T] () at (1,0) {};        
          \node[T,label=above:$s_1$] () at (2,2) {}; 
          \node[txt]() at (1.6,.4){$P_1$}; 
            \node() at (1,-.85){(iv)}; 
                  	\end{tikzpicture}
\end{center}
	\caption{$\pi_1\subset A$ or $\pi_1\subset B$}
	\label{Q5c}
	\end{figure}
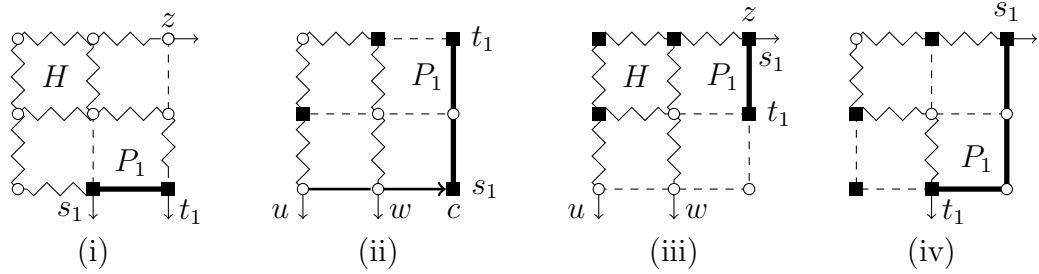

 Case c: $\pi_1\subset A$. For $\|S\|=2,3$, we have $\|B\setminus A\|\leq 1$. We take a linkage $P_1$ in $A$ for $\pi_1$, and consider the graph $H$ in the complement of $P_1$ shown in   Fig.\ref{Q5c}(i). For 
 any three terminals in $S$
 there exist pairwise edge disjoint escape paths to the free vertices $s_1,t_1$ and $z=(1,3)$.
  
Case d:  $\pi_1\subset B$. First we take the linkage $P_1\subset B$ for $\pi_1$.
If  $\|S\|=2$ and there is a terminal in $A\setminus B$, then shift it to $c=(3,3)$  along $A$. In the complement of $P_1$ an AA-clip $CL(u,w)$ is shown in Fig.\ref{Q5c}(ii), where $u,w\in A\setminus B$. For $\|S\|=3$, Fig.\ref{Q5c}(iii) shows a subgraph $H$ with the property that for any  three terminals in $S$, there are pairwise edge disjoint escape paths into the free vertices $u,w$ and $z=(1,3)$.
 
 Case e: $\|S\|=3$ or $2$. If $\pi_1\subset L$ and $\|S\|=3$, we take the linkage $P_1\subset L$ for $\pi_1$, then $H$
 in Fig.\ref{Q5c}(iii) can be used to escape. 
For $\|S\|=2$ we may assume, by Cases c and d, that $s_1\in A\setminus B$, $t_1\in B\setminus A$. The AA-clip 
 in Fig.\ref{Q5c}(ii) works if both vertices of $A\setminus B$
 are free. If it is not the case, then an AB-clip $CL(s_1,t_1)$
  is defined in Fig.\ref{Q5c}(iv).

 \begin{figure}[htp]
  \tikzstyle{T} = [rectangle, minimum width=.1pt, fill, inner sep=2.5pt]
\tikzstyle{B} = [circle, draw=black!, minimum width=1pt, fill=white, inner sep=1.5pt]
\tikzstyle{txt}  = [circle, minimum width=1pt, draw=white, inner sep=0pt]
\tikzstyle{Wedge} = [draw,line width=1.5pt,-,black!100]
\tikzstyle{M} = [circle, draw=black!, minimum width=1pt, fill=white, inner sep=1pt]
\begin{center}
		\begin{tikzpicture}
           
\draw[dashed]  (2,2) --(3,2)--(3,1) (1,3)--(2,3) (1,1)--(2,1); 
 \draw[snake]  (2,3)--(2,1.1) (1,3)--(1,1.1) (2,1)--(2.9,1) (1,2)--(2,2);
 \draw[->]  (1,1) -- (1,.6); \draw[->]  (3,1) -- (3,.6);
  
   \draw[line width=2pt] (2,3) --(3,3)--(3,2);  
         
   \foreach \x in{1,2,3}  \foreach \y in{1,2,3} \node[B]() at (\x,\y){};      
                         
  \node()at(3.2,.8){$w$};       \node()at(.75,.8){$u$}; 
\node[T,label=above:$s_1$]() at (2,3){};    
\node[T,,label=right:$t_1$]() at (3,2){};  
      \node[T]() at (2,1){};   
                               
                \node() at (2,0){(i)};       
            \node() at (1.35,1.5){$P_1$};                       
     	\end{tikzpicture}
	\hskip.5truecm				
	\begin{tikzpicture}
           
\draw[dashed]  (2,3) --(3,3) --(3,2); 
 \draw[snake]  (1,3)--(1,1.1) (1,1)--(1.9,1) (1,3)--(2,3)--(2,1.1) (2,1)--(2.9,1);
 \draw[->]  (2,1) -- (2,.6); \draw[->]  (3,1) -- (3,.6);
  
   \draw[line width=2pt] (1,2) --(3,2)--(3,1) ;  
         
   \foreach \x in{1,2,3}  \foreach \y in{1,2,3} \node[B]() at (\x,\y){};      
                         
 \node[T,label=right:$t_1$]() at (3,1){};   
  \node[B](w) at (2,1){};         \node() at (1.7,.8){$w$};
      \node[T]() at (1,1){};  

\node[T,label=left:$s_1$]() at (1,2){};    
        \node()at(2.75,.8){$u$};                             
                \node() at (2,0){(ii)};       
            \node() at (2.65,1.65){$P_1$};                       
     	\end{tikzpicture}
	\hskip.5truecm	
\begin{tikzpicture}            
\draw[dashed]  (3,2)--(2,2) (2,1)--(3,1)--(3,3); 

\draw[snake] (1,3)--(2,3)--(2.9,3)
;\draw[->] (3,3)--(3.4,3);
\draw[snake] (2,3)--(2,2.1) (1,2)--(2,2)--(2,1.1);
;\draw[->] (2,1)--(2,.6);
\draw[line width=2pt] (1,3)--(1,1) --(2,1);    
      
\foreach \x in {1,2,3}\foreach \y in{1,2,3} \node[B]() at (\x,\y){};                                

            \node[T] () at (3,1) {};     \node () at (2.3,.8) {$t_1$};   
            \node[T] () at (2,1) {};     
            \node[T,label=above:$s_1$] (s1) at (1,3) {};   
             \node[B,label=above:$z$] (z) at (3,3) {};         
                    
\node() at (1.35,1.5){$P_1$};     
      \node() at (2,0){(iii)};     
     	\end{tikzpicture}				
\end{center}
	\caption{$s_1\in S$ and  $\|S\|=3$ }
	\label{4Qe}
\end{figure}
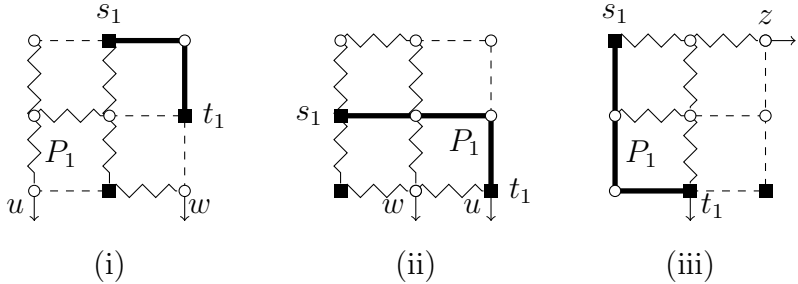

Assume now that $s_1\in S$.  If $\|S\|=2$, let $s_1,s_2$ be the two terminals in $S$, and let $w\in A$ be a free vertex (eventually $w=t_1$). Since $Q$ is weakly $2$-linked, there is an $s_1,t_1$-path $P_1$ and an edge disjoint escape path from $s_2$ to $w$. For $\|S\|=3$, 
let $s_1=(i,j)$, $1\leq i,j\leq 2$. 
If $t_1\in B$, then we take the (unique) $s_1,t_1$-path $P_1\subset B\cup A(i)$; if $t_1\in A\setminus B$ then we take the linkage $P_1\subset A\cup B(j)$ for $\pi_1$. In the first case we define an AA-clip $CL(u,w)$, where $u,w\in A$ are free vertices (see Fig.\ref{4Qe}(i) and (ii)). 
In the second case, after selecting the linkage $P_1$ as described above, the only case when no AA-clip $CL(u,w)$ can be defined is  $s_1\in B(1)$,
$t_1=(3,2)$, furthermore, $(3,1)$ is a free vertex, and $(3,3)$
is a terminal.
Then $z=(1,3)$ is a free vertex, and we may use an AB-clip $CL(t_1,z)$ for escaping the two terminals (see Fig.\ref{4Qe}(iii)).\end{proof}

\end{document}